\documentclass[twoside,11pt]{article}

\usepackage[margin=1in]{geometry}
\usepackage{mathptmx} 
\usepackage{authblk} 
\usepackage{setspace} 
\usepackage{parskip} 
\usepackage{fancyhdr} 
\usepackage{lastpage} 
\usepackage{amsmath,amssymb,amsthm}
\usepackage{bm} 
\usepackage{mathtools} 
\usepackage{physics} 
\usepackage{algorithm} 
\usepackage{algpseudocode}

\theoremstyle{plain}
\newtheorem{theorem}{Theorem}[section]
\newtheorem{lemma}[theorem]{Lemma}

\theoremstyle{definition}
\newtheorem{definition}[theorem]{Definition}

\newtheorem{example}[theorem]{Example}
\newtheorem{remark}[theorem]{Remark}

\usepackage{graphicx}
\usepackage{subcaption} 
\usepackage{booktabs} 
\usepackage{multirow} 
\usepackage{multicol} 
\usepackage{longtable} 
\usepackage{rotating} 
\usepackage{caption} 
\usepackage{float} 
\graphicspath{{figures/}} 
\usepackage[
  backend=biber,       
  sorting=nty,          
  maxnames=3,           
  minnames=1,           
  maxbibnames=5,        
  minbibnames=3,       
  doi=true,             
  isbn=false,           
  url=false,            
  eprint=false          
]{biblatex}
\DeclareFieldFormat{journaltitle}{\textit{#1}}  
\DeclareFieldFormat{booktitle}{\textit{#1}}    
\DeclareFieldFormat{title}{#1.} 

\usepackage{lineno}

\newcommand{\keywords}[1]{\noindent\textbf{Keywords:} #1}

\usepackage[
  colorlinks=true,
  linkcolor=blue,
  citecolor=blue,
  urlcolor=blue,
  pdfauthor={Yiyuan Wang},
  pdftitle={Reinforcement Learning Method for Zero-Sum Linear-Quadratic Stochastic Differential Games in Infinite Horizons},
  pdfsubject={Reinforcement Learning, Zero-Sum Stochastic Differential Games},
  pdfkeywords={Reinforcement Learning, Zero-Sum Stochastic Differential Games}
]{hyperref}
\usepackage{doi} 
\usepackage{url}

\title{Reinforcement Learning Method for Zero-Sum Linear-Quadratic Stochastic Differential Games in Infinite Horizons}
\author{Yiyuan Wang \thanks{Yiyuan Wang is at the Shandong University Zhongtai Securities Institute for Financial Studies, Shandong University, 27 Shanda Nanlu, Jinan, P.R. China, 250100 (Email: wangyiyuan@mail.sdu.edu.cn).}}

\date{}

\begin{document}
\maketitle

\begin{abstract}
In this work, we propose, for the first time, a reinforcement learning framework specifically designed for zero-sum linear-quadratic stochastic differential games. This approach offers a generalized solution for scenarios in which accurate system parameters are difficult to obtain, thereby overcoming a key limitation of traditional iterative methods that rely on complete system information. In correspondence with the game-theoretic algebraic Riccati equations associated with the problem, we develop both semi-model-based and model-free reinforcement learning algorithms by combining an iterative solution scheme with dynamic programming principles. Notably, under appropriate rank conditions on data sampling, the convergence of the proposed algorithms is rigorously established through theoretical analysis. Finally, numerical simulations are conducted to verify the effectiveness and feasibility of the proposed method.\\
\vspace{0.5em}
 \\
\end{abstract}

\section{Introduction} 
\label{sec:introduction}
\quad

The vigorous development of artificial intelligence technology has profoundly reshaped the research paradigms of various disciplines. As a highly representative branch thereof, reinforcement learning (RL) has emerged as a pivotal tool for addressing complex dynamic decision-making problems by virtue of its core mechanisms of trial-and-error exploration and delayed reward optimization, garnering extensive attention from the scientific and engineering communities for a long time. See \cite{Minsky1954,Mendel1970,Sutton2018,Fayaz2022} for details. In the field of control science, RL algorithms have been widely applied to the optimal control of both deterministic and stochastic control systems, offering novel insights into resolving longstanding challenges in traditional control theory such as model unknowns and high-dimensional state spaces. For the latest advances in continuous-time control systems, refer to \cite{Wang2020,Jia2022,Li2022,Pang2023,Xu2023,Liang2024,Ming2024}.

As an important research direction in the field of stochastic control, zero-sum stochastic differential games (ZSSDG) serve as a fundamental theoretical framework for characterizing the adversarial behaviors of agents in continuous time, and have been extensively adopted for the modeling and analysis of various complex engineering scenarios. See \cite{Bacsar2018} for details. Zero-sum linear-quadratic stochastic differential games (ZSLQSDG) constitute a core category of problems in the theory and practice of ZSSDG, as a large number of nonlinear problems can be approximately solved via linear models, relevant theoretical advances can be found in \cite{Mou2006,Sun2014,Yu2015,Sun2016,Sun2021,Zhou2022,Zhang2024,Sun2024}. The linear quadratic approximation assumption for system dynamics and performance indices provides a direct and effective research perspective for an intuitive understanding of ZSSDG mechanisms among agents and the exploration of relevant computational methods, refer to \cite{Wang2025}.

RL algorithms developed on the basis of dynamic programming (DP) constitute an important technical approach and thus provide a key tool for addressing such problems, refer to \cite{Sutton2002, Lewis2013}. Specifically, \cite{Vrabie2011} present an approximate/adaptive dynamic programming (ADP) algorithm that adopts the idea of integral reinforcement learning to online determine the Nash equilibrium solution for the two-player zero-sum differential game with linear dynamics and infinite-horizon quadratic cost. \cite{Wu2013} present a simultaneous policy update algorithms for learning the solution of linear continuous-time $H_\infty$ state feedback control. Subsequently, \cite{Li2014} develop an integral reinforcement learning algorithm based on policy iteration to online learn the Nash equilibrium solution for a two-player zero-sum differential game with completely unknown linear continuous-time dynamics. \cite{Chen2024} propose two online policy-iteration algorithms for solving linear continuous-time $H_\infty$ regulation problems with unknown dynamics. However, all the aforementioned ADP algorithms for linear quadratic differential game (LQDG) problems are designed for deterministic systems, and research on problems involving continuous-time stochastic systems remains relatively scarce. At present, for the complex scenario where noise terms are affected by single control inputs, \cite{Sun2026} proposed RL-based solutions for $H_\infty$ control problems—a subclass of the ZSLQSDG problem. 

ADP for continuous-time systems entails greater inherent challenges compared with discrete-time systems, implying that numerous ADP algorithms developed for discrete-time systems cannot be directly extended to the continuous-time setting. The core of the ZSLQSDG problem consists in solving closed-loop strategy pairs that satisfy the Nash equilibrium, which inevitably requires tackling highly coupled nonlinear game-theoretic algebraic Riccati equations (GTARE). For the more complex coupled case where noise terms are simultaneously affected by the control inputs of both players, a recent work by \cite{Wang2025} has introduced an iterative framework to address this problem. However, conventional iterative approaches become largely infeasible when complete knowledge of system parameters is unavailable.

This paper develops, for the first time, a comprehensive RL algorithmic framework specifically designed for the ZSLQSDG problem, and provides a unified, general solution applicable to scenarios where accurate system parameter information is difficult to acquire. In particular, customized RL algorithms are constructed to handle distinct cases corresponding to different degrees of prior knowledge about system dynamics. Most importantly, under the condition that data sampling satisfies the required rank condition, this paper rigorously derives and proves the convergence of the proposed algorithms, thus filling the theoretical gap in the convergence analysis of RL methods for ZSLQSDG problems.

The remainder of this paper is structured as follows. Section \ref{sec:preliminary} introduces the mathematical formulation of the ZSLQSDG problem and related preliminary background. Section \ref{sec:RL} elaborates on the design of on-policy and off-policy semi-model-based RL algorithms, as well as fully model-free RL algorithms, all established within the dynamic programming paradigm. Section \ref{sec:ConvergenceAnalysis} presents a rigorous convergence analysis. Finally, Section \ref{sec:simulation} provides a detailed numerical simulation example to validate the effectiveness of the proposed approach.

\section{Preliminary} 
\label{sec:preliminary}

\subsection{Notation}
Let us first introduce the following notation:
\begin{itemize}
    \item $\mathbb{R}^n$: $n$-dimensional real Euclidean space; $\mathbb{C}^{-}$: the set of complex numbers with negative real part; $\mathbb{R}^{n\times m}$: the space of $n\times m$ real matrices; $\mathbb{S}^n$: the set of all $n\times n$ symmetric matrices; $\overline{\mathbb{S}}_{+}^n$: the set of all $n\times n$ symmetric positive semi-definite matrices; $\mathbb{S}_{+}^n$: the set of all $n\times n$ symmetric positive definite matrices. 
    \item $I_n$: the identity matrix of size $n$; $\mathbb{O}_{n \times m}$: the null matrix of size $n \times m$. It can be simplified as 0 when no ambiguity is generated; $A_{m \times n}^{\times (l)}$ denotes the object represented by $A_{m \times n}$ arranged repeatedly for $l$ times.
    \item $A^{\top}$: the transpose of the matrix $A$; $\langle\cdot,\cdot\rangle$: the inner product on a Hilbert space; \(\otimes\): Kronecker product; If $A\in\mathbb{S}_{+}^n$ (resp., $A\in\overline{\mathbb{S}}_{+}^n$), we write $A\succ0$ (resp., $A\succeq0$). For any $A,B\in\mathbb{S}^n$, we use the notation $A\succ B$ (resp., $A\succeq B$) to indicate that $A - B\succ0$ (resp., $A - B\succeq0$).
    \item Let $\mathbb{H}$ be a Euclidean space, and we define the following space \footnote{$\varphi \in \mathbb{F}$ denotes that $\varphi$ is $\mathbb{F}$-progressively measurable.}: 
    $L_{\mathbb{F}}^2(\mathbb{H}) = \{\varphi: [0,\infty) \times \Omega \rightarrow \mathbb{H} \mid \varphi \in \mathbb{F},\mathbb{E} \int_{0}^{\infty} |\varphi(t)|^2 \, dt < \infty \} $; $\mathcal{U}_i = L_{\mathbb{F}}^2(\mathbb{R}^{m_i})(i = 1,2)$, $\mathcal{X}_{\mathrm{loc}}[0,\infty) = \{\varphi: [0,\infty) \times \Omega \rightarrow \mathbb{R}^n \mid \varphi \in \mathbb{F} \,\text{ is continuous},\, \mathbb{E}\left[\sup_{0 \leq t \leq T} |\varphi(t)|^2\right] < \infty, \forall\, T > 0 \}$; $\mathcal{X}[0,\infty) = \{\varphi \in \mathcal{X}_{\mathrm{loc}}[0,\infty) \mid \mathbb{E} \int_{0}^{\infty} |\varphi(t)|^2 \, dt < \infty \}$.
\end{itemize}

\subsection{Zero-Sum Linear-Quadratic Stochastic Differential Games in Infinite Horizons}

Let $(\Omega,\mathcal{F},\mathbb{F},\mathbb{P})$ be a complete filtered probability space on which a r-dimensional standard Brownian motion $W = \{W(t)^{\top}=(w_{1}(t),\dotsb,w_{r}(t));t\geq0\}$ is defined with $\mathbb{F}=\{\mathcal{F}_t\}_{t\geq0}$ being the usual augmentation of the natural filtration generated generated by $W$. We consider the following controlled linear stochastic differential equation (SDE) on $[0,\infty)$:
\begin{equation} 
\label{zslqsdg:sde}
    \begin{cases} 
        &dX(t) = \lbrack AX(t) + B_1 u_1(t) + B_2 u_2(t) \rbrack dt \\
        &\quad\quad\quad+ \sum_{l=1}^{r}\lbrack C_{l}X(t) + D_{l,1} u_1(t) + D_{l,2} u_2(t) \rbrack dw_{l}(t) \\
        &X(0) = x
    \end{cases} ,
\end{equation}
in which $x \in\mathbb{R}^n$, $A,C_{l}\in\mathbb{R}^{n\times n}$ and $B_i,D_{l,i}\in\mathbb{R}^{n\times m_i}$ ($i = 1,2;1 \leq l \leq r$). In the above, the state process $X$ is an $n$-dimensional vector, player 1 $u_1$ and player 2 $u_2$ are an $m_1$-dimensional vector and an $m_2$-dimensional vector, respectively. 

In this paper, Player 1 and Player 2 share the same performance functional:
\begin{equation}
\label{zslqsdg:performance_functional}
    J(x;u_1,u_2)=\mathbb{E}\int_{0}^{\infty}\Bigg[\Bigg\langle\begin{pmatrix}
    Q & S_1^{\top} & S_2^{\top}\\
    S_1 & R_{11} & R_{12}\\
    S_2 & R_{21} & R_{22}
    \end{pmatrix}\begin{pmatrix}
    X(t)\\
    u_1(t)\\
    u_2(t)
    \end{pmatrix},\begin{pmatrix}
    X(t)\\
    u_1(t)\\
    u_2(t)
    \end{pmatrix}\Bigg\rangle
    \Bigg]dt,
\end{equation}
the weighting coefficients in \eqref{zslqsdg:performance_functional} satisfy:
\[
    Q\in\mathbb{S}^n,\quad R_{21}^{\top}=R_{12}\in\mathbb{R}^{m_1\times m_2},\quad S_i\in\mathbb{R}^{m_i\times n}, R_{ii}\in\mathbb{S}^{m_i}(i = 1,2).
\]
We define the set of admissible controls as $\mathcal{U}_{ad}(x)=\{ (u_1,u_2)\in\mathcal{U}_1 \times \mathcal{U}_2\mid X(\cdot;x,u_1,u_2)\in\mathcal{X}[0,\infty)\}$. A pair $(u_1,u_2)\in\mathcal{U}_{ad}(x)$ is called an admissible control pair for the initial state $x$, and the corresponding $X(\cdot;x,u_1,u_2)$ is referred to as the admissible state process. In this case, $J(x, u_1, u_2)$ is clearly well-defined.

In this zero-sum game, Player 1 (\textit{the maximizer}) selects control $u_1$ to maximize \eqref{zslqsdg:performance_functional}, while Player 2 (\textit{the minimizer}) chooses $u_2$ to minimize the same function. The problem is to find an admissible control pair $(u_1^*,u_2^*)$ that both players can accept. We denote the above-mentioned problem as $({SDG})^{\,0}_{\infty}$ for short. For a description of the $({SDG})^{\,0}_{\infty}$ problem, refer to \cite{Sun2020_book} . More detailed information can be found therein. 

Theorem 2.6.7 in \cite{Sun2020_book} directly establishes the relationship between $({SDG})^{\,0}_{\infty}$ and the stochastic GTARE. 
The $({SDG})^{\,0}_{\infty}$ problem corresponds to following stochastic GTARE:
\begin{equation}
\label{zslqsdg:gtare}
    \begin{cases}
        Q(P)-S(P)^{\top}R(P)^{-1}S(P)=0\\
        \mathcal{R}(S(P))\subseteq \mathcal{R}(R(P))\\
        R_{11}(P)\preceq 0,\quad R_{22}(P) \succeq 0 
    \end{cases},    
\end{equation}
where $\mathcal{R}(M)$ denotes the range of a matrix $M$,
\begin{equation}
\label{zslqsdg:gtare_auxiliarymatrix}
\begin{aligned}
    &Q(P)=PA + A^{\top}P+\sum_{l=1}^{r}C_{l}^{\top}PC_{l} + Q,\\
    &S(P)=\begin{bmatrix}
    S_1(P) \\
    S_2(P)
    \end{bmatrix}=\begin{bmatrix}
    B_1^{\top}P+\sum_{l=1}^{r}D_{l,1}^{\top}PC_{l}+S_1 \\
    B_2^{\top}P+\sum_{l=1}^{r}D_{l,2}^{\top}PC_{l}+S_2\\
    \end{bmatrix},\\
    &R(P)=\begin{bmatrix}
    R_{11}(P) & R_{12}(P)\\
    R_{21}(P) & R_{22}(P)\\
    \end{bmatrix},\quad R_{ij}(P)=R_{ij}+\sum_{l=1}^{r}D_{l,i}^{\top}PD_{l,j}\,(i,j = 1,2).\\
\end{aligned}
\end{equation}

\begin{definition}
\label{zslqsde:mean_square}
The system 
\begin{equation*}
    \begin{cases} 
        dX(t) =  A X(t) dt + \sum_{l=1}^{r} C_{l}X(t) dw_{l}(t) \\
        X(0) = x
    \end{cases},
\end{equation*}
donated as $(A,C_{1},\dotsb,C_{r})$ is called \textit{mean-square stable} such that the solution 
satisfies \[\lim_{t\to\infty} \mathbb{E}[X(t)^{\top}X(t)] = 0 \,\,\text{for every initial state}\,\, x \in \mathbb{R}^n.\]
\end{definition}

A concept related to the above admissible control pair is that of \textit{mean-square stabilizers}. We define the set of all such stabilizers associated with the SDE \eqref{zslqsdg:sde} as
\begin{align*}
&\mathcal{K}=\mathcal{K}\left[ A,C_{1},\cdots,C_{r}; \begin{bmatrix} B_1&B_2\end{bmatrix},\begin{bmatrix}D_{1,1}&D_{1,2}\end{bmatrix},\cdots, \begin{bmatrix}D_{r,1}&D_{r,2}\end{bmatrix}\right] \\
&= \left\{
    \begin{aligned}
        &\text{All of} \,\,\,(\varTheta_1, \varTheta_2) \in \mathbb{R}^{m_1 \times n} \times \mathbb{R}^{m_2 \times n} \,\text{such that the system} \,(A + B_1 \varTheta_1 + B_2 \varTheta_2,\\
        &C_1 +D_{1,1} \varTheta_1 + D_{1,2} \varTheta_2, \dotsb, C_r + D_{r,1} \varTheta_1 + D_{r,2} \varTheta_2) \text{ is mean-square stable}.
    \end{aligned}
\right\}.
\end{align*}

\begin{definition}[\cite{Sun2020_book}]
A matrix $P\in\mathbb{S}^n$ is called a stabilizing solution of stochastic GTARE \eqref{zslqsdg:gtare} if $P$ satisfies \eqref{zslqsdg:gtare} and $(K_1(P), K_2(P)) \in \mathcal{K}$, where $\begin{bmatrix} K_1(P) \\ K_2(P) \end{bmatrix}=-R(P)^{-1}\begin{bmatrix}S_1(P) \\S_2(P)\end{bmatrix}$ and $\mathcal{K}$ is the set of all such stabilizers associated with the SDE \eqref{zslqsdg:sde}.
\end{definition}

To facilitate subsequent analysis of the problem, we introduce the following decoupling representation method. Setting 
\[
\begin{bmatrix}  v_1(t) \\ v_2(t) \end{bmatrix}=\begin{bmatrix}  u_1(t) \\ u_2(t) \end{bmatrix}-\begin{bmatrix} K_1(0) \\ K_2(0) \end{bmatrix},\quad 
\begin{bmatrix} K_1(0) \\ K_2(0) \end{bmatrix} = -R(0)^{-1} \begin{bmatrix} S_1(0) \\ S_2(0) \end{bmatrix},
\] and $v_2(t) = Lx(t)$ in $({SDG})^{\,0}_{\infty}$, we obtain
\begin{equation*} 
    \begin{cases} 
        dX(t) = \lbrack A_{L}X(t) + B_1v_1(t) \rbrack dt + \sum_{l=1}^{r}\lbrack C_{lL}X(t) + D_{l,1}v_1(t) \rbrack dw_{l}(t) \\
        X(0) = x
    \end{cases}
\end{equation*}
in which $x \in\mathbb{R}^n$, and 
\begin{equation*}
    J_{L}(x;v_1)=\mathbb{E}\int_{0}^{\infty}\Bigg[\Bigg\langle\begin{pmatrix}
    Q_{L} & S^{\top}_{L} \\
    S_{L} & R_{11}
    \end{pmatrix}\begin{pmatrix}
    X(t)\\
    v_1(t)
    \end{pmatrix},\begin{pmatrix}
    X(t)\\
    v_1(t)
    \end{pmatrix}\Bigg\rangle
    \Bigg]dt,
\end{equation*}
where
\[
\begin{cases}
&A_{L} = A + B_1K_1(0) + B_2K_2(0)+ B_2L\\
&C_{lL} = C_{l} + D_{l,1}K_1(0) + D_{l,2}K_2(0)+ D_{l,2}L,  1 \leq l \leq r \\
&Q_{L} = Q - S^{\top}(0) R(0)^{-1}S(0) + L^{\top}R_{22}L,\quad S_{L}= R_{12}L 
\end{cases}.
\]
The corresponding ARE is
\begin{equation}
\label{gtare_lw}
\begin{aligned}
&PA_{L} + A_{L}^{\top}P+ \sum_{l=1}^r C_{lL}^{\top}PC_{lL}+ Q_{L} - (B_{1}^{\top}P + \sum_{l=1}^r D_{l,1}^{\top}PC_{lL} + S_{L}) ^{\top} \\
& \quad \times(R_{11} + \sum_{l=1}^r D_{l,1}^{\top}PD_{l,1})^{-1} ( B_1^{\top}P + \sum_{l=1}^r D_{l,1}^{\top}PC_{lL} + S_{L} ) = 0.
\end{aligned} 
\end{equation}
Throughout this work, $\mathcal{A}$ stands for the set of $L$, where $L \in \mathbb{R}^{m_2 \times n} $ satisfy: 
\begin{itemize}
    \item The system $(A_{L},C_{1L},\dotsb,C_{rL})$ is \textit{mean-square stable}.
    \item The corresponding ARE \eqref{gtare_lw} has a stabilizing solution $\tilde{P}_{L}$, satisfying the sign conditions
    \[R_{11} + \sum_{l=1}^r D_{l,1}^{\top}\tilde{P}_{L}D_{l,1} \prec 0.\]
\end{itemize}

\begin{definition}[\cite{Dragan2013book}]
Given the following stochastic observation system:
\begin{equation*}
    \begin{cases}
         dX(t) = A X(t)dt + \sum_{l=1}^{r} C_l X(t)dw_l(t) \\
         dY(t) = E_0 X(t)dt + \sum_{l=1}^{r} E_l X(t)dw_l(t)
    \end{cases},
\end{equation*}
where $E_l \in R^{ q\times n}(0 \leq l \leq r)$, we denote this system by $\left[ E_0, E_1, \dotsb, E_r; A, C_1, \dotsb, C_r \right]$. It is said to be \textit{stochastically detectable} if there exists a constant matrix $\varTheta \in \mathbb{R}^{n \times q}$ such that the system
$( A + \varTheta E_0,\ C_1 + \varTheta E_1, \dotsb, C_r + \varTheta E_r )$
is \textit{mean-square stable}.
\end{definition}

\section{Reinforcement Learning}
\label{sec:RL}

\subsection{Nested Iterative}
\label{subsec:AlgorithmDesign_1}

Before introducing the subsequent algorithm, we first define the parameter update schemes required for the Algorithm \ref{alg1:Nested_Iteration}, \ref{alg2:SemiModelBased} and \ref{alg3:SemiModelBased} as follows:
\begin{equation}
\label{iterative:matrices_evolve} 
\begin{aligned}
   &R_{(k)}=\begin{bmatrix}
    R_{11,(k)}& R_{12,(k)}\\
    R_{21,(k)}& R_{22,(k)}\\
    \end{bmatrix}=\begin{bmatrix}
    R_{11}+\sum_{l=1}^{r}D_{l,1}^{\top}P^{(k)}D_{l,1} & R_{12}+\sum_{l=1}^{r}D_{l,1}^{\top}P^{(k)}D_{l,2}\\
    R_{21}+\sum_{l=1}^{r}D_{l,2}^{\top}P^{(k)}D_{l,1} & R_{22}+\sum_{l=1}^{r}D_{l,2}^{\top}P^{(k)}D_{l,2}\\
    \end{bmatrix},\\
    &R^{\sharp}_{22,(k)}=R_{11,(k)}-R_{12,(k)}R_{22,(k)}^{-1}R_{21,(k)},\,S_{(k)} = \begin{bmatrix} S_{1,(k)} \\ S_{2,(k)} \end{bmatrix} =\begin{bmatrix} B_1^{\top} P^{(k)} + \sum_{l=1}^{r}D_{l,1}^{\top} P^{(k)}C_{l}+S_1 \\ B_2^{\top} P^{(k)} + \sum_{l=1}^{r}D_{l,2}^{\top} P^{(k)} C_{l} +S_2 \end{bmatrix},\\
    &\,\text{and}\, \begin{bmatrix} N_{1,(k)} \\ N_{2,(k)} \end{bmatrix}=\begin{bmatrix} B_1^{\top} Z^{(k)} + \sum_{l=1}^{r}D_{l,1}^{\top} Z^{(k)} C_{l,(k)} \\ B_2^{\top} Z^{(k)} + \sum_{l=1}^{r}D_{l,2}^{\top} Z^{(k)} C_{l,(k)}  \end{bmatrix}.\\
\end{aligned}
\end{equation}

To lay the theoretical groundwork for the subsequent on-policy and off-policy reinforcement learning algorithms, we first introduce a baseline nested iteration Algorithm \ref{alg1:Nested_Iteration} under the premise of fully known problem parameters.

\begin{algorithm}[htbp]
\caption{Nested Iteration}
\label{alg1:Nested_Iteration}
    \begin{algorithmic}[1]
        \State \textbf{Input}: Problem parameters.
        \State \textbf{Initialization}: 
        \State \quad Set outer iteration counters $k = 0$; convergence tolerance $\varepsilon > 0$. 
        \State \quad Initialize matrices $L \in \mathcal{A}$, $P^{(0)} = 0$, $R_{22,(0)}=R_{22}$, $M_{(0)}=Q-S^{\top}(0) R^{-1}(0)S(0)$, $\begin{bmatrix} K_{1}^{(0)} \\ K_{2}^{(0)} \end{bmatrix}=-R^{-1}(0)S(0)$ and iterative parameter matrices $A_{(0)}=A+\begin{bmatrix} B_1 & B_2 \end{bmatrix}\begin{bmatrix} K_{1}^{(0)} \\ K_{2}^{(0)} \end{bmatrix}$, $C_{l,(0)}=C_{l}+\begin{bmatrix} D_{l,1} & D_{l,2} \end{bmatrix}\begin{bmatrix} K_{1}^{(0)} \\ K_{2}^{(0)} \end{bmatrix}(l=1,\cdots,r)$.
        \Repeat
            \State \quad Reset inner counter $j = 0$; Initialize gain matrix $K_2^{(0,0)}=L$ or $K_2^{(k,0)}=K_2^{(0)}+K_2^{(0,0)}-K_2^{(k-1)}(k=1,2,\cdots)$.
                \Repeat
                    \State \quad \textbf{Policy Evaluation}: Solve the algebraic Riccati equation for $Z^{(k,j+1)}$:
                    \begin{equation}
                    \label{alg1:PolicyEvaluation}
                    \begin{aligned}
                        &Z^{(k,j+1)}(A_{(k)} + B_2 K_2^{(k,j)}) + (A_{(k)} + B_2 K_2^{(k,j)})^\top Z^{(k,j+1)} \\
                        &+ \sum_{l=1}^r (C_{l,(k)} + D_{l,2} K_2^{(k,j)})^\top Z^{(k,j+1)} (C_{l,(k)} + D_{l,2} K_2^{(k,j)}) \\
                        &+ M_{(k)} + K_2^{(k,j)\top} R_{22,(k)} K_2^{(k,j)} = 0.
                    \end{aligned}
                    \end{equation}
                    \State \quad \textbf{Policy Improvement}: Update the gain matrix $K_2^{(k,j)}$ to $K_2^{(k,j+1)}$:
                    \begin{equation}
                    \label{alg1:PolicyImprovement}
                    \begin{aligned}
                        K_2^{(k,j+1)} = - ( R_{22,(k)} + \sum_{l=1}^r D_{l,2}^\top Z^{(k,j+1)} D_{l,2} )^{-1} (B_2^\top Z^{(k,j+1)}  + \sum_{l=1}^r D_{l,2}^\top Z^{(k,j+1)} C_{l,(k)}).
                    \end{aligned}
                    \end{equation}
                    \State \quad Increment inner counter: $j \gets j + 1$.
                \Until{$\| Z^{(k,j)} - Z^{(k,j-1)} \| < \varepsilon$}
            \State \quad \begin{equation}\label{alg1:outer_iteration}\text{Update} \,\,Z^{(k)} =Z^{(k,j)},\,P^{(k+1)}= P^{(k)} + Z^{(k)}.\end{equation}
            \State \quad Evolve matrices $S_{(k+1)}, N_{1,(k)}, N_{2,(k)}, R_{(k+1)}, R_{12,(k+1)}, R_{22,(k+1)}$, and $R^{\sharp}_{22,(k+1)}$ via \eqref{iterative:matrices_evolve} by using $P^{(k+1)}$ and $Z^{(k)}$. 
            \State \quad Update $\begin{bmatrix} K_{1}^{(k+1)} \\ K_{2}^{(k+1)} \end{bmatrix} =-R_{(k+1)}^{-1}S_{(k+1)} $, $A_{(k+1)} = A +  \begin{bmatrix} B_1 & B_2 \end{bmatrix} \begin{bmatrix} K_{1}^{(k+1)} \\ K_{2}^{(k+1)} \end{bmatrix}, C_{l,(k+1)} = C_{l} +  \begin{bmatrix} D_{l,1} & D_{l,2} \end{bmatrix} \begin{bmatrix} K_{1}^{(k+1)} \\ K_{2}^{(k+1)} \end{bmatrix}(1 \leq l \leq r)$ and $M_{(k+1)} =-\bigl(N_{1,(k)}-R_{12,(k+1)}R_{22,(k+1)}^{-1}N_{2,(k)}\bigr)^{\top} R^{\sharp-1}_{22,(k+1)}\bigl(N_{1,(k)}-R_{12,(k+1)}R_{22,(k+1)}^{-1}N_{2,(k)}\bigr)$.
            \State \quad Increment outer counter: $k \gets k + 1$.
        \Until{$\| P^{(k)} - P^{(k-1)} \| < \varepsilon$}
        \State \Return $P^{(k)}$
    \end{algorithmic}
\end{algorithm}

\begin{remark}
This nested iteration scheme is derived from the unified computational framework for addressing this class of problems, as proposed in \cite{Wang2025}. By establishing theoretical guarantees for the scheme’s core properties, the aforementioned work lays a solid foundation for the development of reinforcement learning approach in the present study.
\end{remark}

To streamline subsequent exposition, we first establish the notational conventions as follows. For a symmetric matrix $P \in \mathbb{S}^n$, define an operator $\text{svec}(P) \in \mathbb{R}^{\frac{1}{2}n(n+1)}$ as:
\[
\text{svec}(P) := \left[ p_{11},\, 2p_{12},\,\dotsb,\, 2p_{1n},\, p_{22},\, 2p_{23},\,\dotsb,\, 2p_{n-1,n},\, p_{nn} \right]^\top
\]
From \cite{Murray2002,Li2022}, there exists a matrix $\mathcal{T} \in \mathbb{R}^{n^2 \times \frac{1}{2}n(n+1)}$ (with rank $\frac{1}{2}n(n+1)$) such that $ \text{vec}(P) = \mathcal{T} \times\text{svec}(P) $, where $\text{vec}(P)$ denotes the standard column-wise vectorization of $P$.  
For a symmetric matrix, only $\frac{1}{2}n(n+1)$ variables are actually required to characterize it. To reduce the degrees of freedom for estimation in subsequent algorithm design while ensuring the symmetry of the matrix, we define the following operator.
For any $W \in \mathbb{R}^{n}$, define $\bar{W} = ( W^\top \otimes W^\top ) \times\mathcal{T} \in \mathbb{R}^{1 \times \frac{1}{2}n(n+1)}$, which satisfies $(W^\top \otimes W^\top) \times \text{vec}(P) = (W^\top \otimes W^\top) \times \mathcal{T} \times \text{svec}(P)= \bar{W} \times \text{svec}(P)$.
Further, we define following sampling matrix for the Algorithm \ref{alg1:Nested_Iteration} and \ref{alg2:SemiModelBased}:
\begin{equation}
\label{alg23:CollectedData}
\begin{aligned}
&\delta_{xx}^{i,(k,j)} = \mathbb{E}^{\mathcal{F}_{t_i}} \left[ \bar{X}^{i,(k,j)\top}(t_{i+1}) - \bar{X}^{i,(k,j)\top}(t_i) \right],\delta_{v_2v_2}^{i,(k,j)} = \mathbb{E}^{\mathcal{F}_{t_i}} \left[ \int_{t_i}^{t_{i+1}} \bar{v}_2^{i,(k,j)}(s)ds \right],\\
&\mathcal{I}_{xx}^{i,(k,j)} = \mathbb{E}^{\mathcal{F}_{t_i}} \left[ \int_{t_i}^{t_{i+1}} X^{i,(k,j)\top}(s) \otimes X^{i,(k,j)\top}(s)ds \right],\mathcal{I}_{xv_2}^{i,(k,j)} = \mathbb{E}^{\mathcal{F}_{t_i}} \left[ \int_{t_i}^{t_{i+1}} X^{i,(k,j)\top}(s) \otimes v_2^{i,(k,j)\top}(s)ds \right],\\
&\mathcal{I}_{v_2v_2}^{i,(k,j)} := \mathbb{E}^{\mathcal{F}_{t_i}} \left[ \int_{t_i}^{t_{i+1}} v_2^{i,(k,j)\top}(s) \otimes v_2^{i,(k,j)\top}(s)ds \right].
\end{aligned}
\end{equation}

\subsection{On-Policy Semi-Model-Based Reinforcement Learning}
\label{subsec:AlgorithmDesign_2}

The convergence analysis of Algorithm \ref{alg1:Nested_Iteration} is deferred to Section \ref{sec:ConvergenceAnalysis}, where a unified convergence analysis for the proposed algorithm is established. For the moment, we focus on the reinforcement learning formulation of the problem and present the following Algorithm \ref{alg2:SemiModelBased}.  

\begin{algorithm}[htbp]
\caption{On-Policy Semi-Model-Based Reinforcement Learning}
\label{alg2:SemiModelBased}
    \begin{algorithmic}[1]
        \State \textbf{Initialization}: 
        \State \quad Set outer iteration counters $k = 0$; convergence tolerance $\varepsilon > 0$. 
        \State \quad Initialize matrices $L \in \mathcal{A}$, $P^{(0)} = 0$, $R_{22,(0)}=R_{22}$, $M_{(0)}=Q-S^{\top}(0)R^{-1}(0)S(0)$, $\begin{bmatrix} K_{1}^{(0)} \\ K_{2}^{(0)} \end{bmatrix}=-R^{-1}(0)S(0)$.
        \Repeat
            \State \quad Reset inner counter $j = 0$; Initialize gain matrix $K_2^{(0,0)}=L+K_{2}^{(0)}$ or $K_2^{(k,0)}=K_2^{(0,0)}-K_2^{(k-1)}(k=1,2,\cdots)$.
            \Repeat
                \State \textbf{Policy Evaluation}:
                \State 1. Obtain $X^{(k,j)}(\cdot)$ by running the SDE \eqref{zslqsdg:sde} with $\begin{bmatrix} u_1^{(k,j)}(\cdot)\\ u_2^{(k,j)}(\cdot) \end{bmatrix}=\begin{bmatrix} K_{1}^{(k)}X(\cdot) \\ K_{2}^{(k)}X(\cdot)+K_{2}^{(k,j)}X(\cdot) \end{bmatrix}$ on $[t, t+\Delta t]$. Compute $\delta_{xx}^{i,(k,j)}$ and $\mathcal{I}_{xx}^{i,(k,j)}(i=1,2,\cdots,N)$.
                \State 2. Solve $Z^{(k,j+1)}$ from the identity \eqref{alg2:PolicyEvaluation+}.
                \State \textbf{Policy Improvement}: Update the gain matrix $K_2^{(k,j)}$ to $K_2^{(k,j+1)}$:
                \begin{equation}
                \label{alg2:PolicyImprovement}
                \begin{aligned}
                        K_2^{(k,j+1)} = -\bigl( R_{22,(k)} + \sum_{l=1}^r D_{l,2}^\top Z^{(k,j+1)} D_{l,2} \bigr)^{-1} \bigl( B_2^\top Z^{(k,j+1)}  + \sum_{l=1}^r D_{l,2}^\top Z^{(k,j+1)} C_{l,(k)}\bigr).
                \end{aligned}
                \end{equation}
                \State \quad Increment inner counter: $j \gets j + 1$.
            \Until{$\| Z^{(k,j)} - Z^{(k,j-1)} \| < \varepsilon$}
            \State \quad Update $Z^{(k)} =Z^{(k,j)},\,P^{(k+1)}= P^{(k)} + Z^{(k)}$. 
            \State \quad Evolve matrices $S_{(k+1)}, N_{1,(k)}, N_{2,(k)}, R_{(k+1)}, R_{12,(k+1)}, R_{22,(k+1)}$, and $R^{\sharp}_{22,(k+1)}$ via \eqref{iterative:matrices_evolve} by using $P^{(k+1)}$ and $Z^{(k)}$. 
            \State \quad Update $\begin{bmatrix} K_{1}^{(k+1)} \\ K_{2}^{(k+1)} \end{bmatrix} =-R_{(k+1)}^{-1} S_{(k+1)} $, $C_{l,(k+1)} = C_{l} +  \begin{bmatrix} D_{l,1} & D_{l,2} \end{bmatrix} \begin{bmatrix} K_{1}^{(k+1)} \\ K_{2}^{(k+1)} \end{bmatrix}(l=1,2,\cdots,r)$, and $M_{(k+1)} =-\bigl(N_{1,(k)}-R_{12,(k+1)}R_{22,(k+1)}^{-1}N_{2,(k)}\bigr)^{\top} R^{\sharp-1}_{22,(k+1)}\bigl(N_{1,(k)}-R_{12,(k+1)}R_{22,(k+1)}^{-1}N_{2,(k)}\bigr)$.
            \State \quad Increment outer counter: $k \gets k + 1$.
        \Until{$\| P^{(k)} - P^{(k-1)} \| < \varepsilon$} 
        \State \Return $P^{(k)}$
    \end{algorithmic}
\end{algorithm}

Next, we examine the relationship between the policy evaluation step in Algorithm \ref{alg2:SemiModelBased} and that in Algorithm \ref{alg1:Nested_Iteration}.

\begin{lemma}
\label{alg2:lemma}
For any $Z^{(k,j+1)}$ and $K_2^{(k,j)}(j=0,1,\cdots)$ generated from \eqref{alg1:PolicyEvaluation} and \eqref{alg1:PolicyImprovement} in Algorithm \ref{alg1:Nested_Iteration} satisfy
\begin{equation}
\label{alg2:PolicyEvaluation}
\delta^{i,(k,j)}_{xx}  \text{svec}(Z^{(k,j+1)})  = -\mathcal{I}^{i,(k,j)}_{xx}\text{vec}(M_{(k)}+ K_2^{(k,j)\top} R_{22,(k)} K_2^{(k,j)})  
\end{equation}
\end{lemma}

\begin{proof}
    From Lemma 3 of \cite{Li2022}, we obtain solving the policy evaluation step in Algorithm \ref{alg2:SemiModelBased} is equivalent to solving the identity:
    \begin{align*}
    &\mathbb{E}^{\mathcal{F}_t}\left[ X^{(k,j)\top}(t+\Delta t) Z^{(k,j+1)} X^{(k,j)}(t+\Delta t) \right]-x^\top Z^{(k,j+1)} x= \\
    &-\mathbb{E}^{\mathcal{F}_t} \int_{t}^{t+\Delta t} X^{(k,j) \top}(s) ( M_{(k)} + K_2^{(k,j)\top} R_{22,(k)} K_2^{(k,j)} ) X^{(k,j)}(s) ds.
    \end{align*}
    We vectorize both sides of the equation simultaneously, and thus obtain the \eqref{alg2:PolicyEvaluation}. 
    The proof details is consistent with that in Lemma 3 of \cite{Li2022}. The difference is that since \[M_{(k)} + K^{(k,j)\top} R_{22,(k)} K^{(k,j)} \succeq 0,\] Lyapunov recursion \eqref{alg1:PolicyEvaluation} has a solution $Z^{(k,j+1)} \in \overline{\mathbb{S}}_{+}^n$. This difference constitutes the key problem to be overcome in the subsequent Lemma \ref{alg:lemma}.
\end{proof}

\begin{remark}
\label{alg2:remark}
When there exist $N$ sampling data points such that the sampling matrix $\delta^{(k,j)}_{xx}$ has full column rank, we have
\begin{equation}
\label{alg2:PolicyEvaluation+}
\begin{aligned}   
\text{svec}(Z^{(k,j+1)}) = -(\delta^{(k,j)\top}_{xx}\delta^{(k,j)}_{xx})^{-1}\delta^{(k,j)\top}_{xx}   \mathcal{I}^{(k,j)}_{xx}
\text{vec}(M_{(k)}+ K_2^{(k,j)\top} R_{22,(k)} K_2^{(k,j)}),
\end{aligned}
\end{equation}
where 
\[
\begin{aligned}
\delta^{(k,j)}_{xx} = \left[ \delta_{xx}^{1,(k,j)\top}, \dotsb, \delta_{xx}^{N,(k,j)\,\top} \right]^\top,\mathcal{I}^{(k,j)}_{xx} = \left[ \mathcal{I}_{xx}^{1,(k,j)\,\top}, \dotsb, \mathcal{I}_{xx}^{N,(k,j)\,\top} \right]^\top.
\end{aligned}
\]
This means that solving the policy evaluation step in Algorithm \ref{alg2:SemiModelBased} is equivalent to solving the \eqref{alg1:PolicyEvaluation} in Algorithm \ref{alg1:Nested_Iteration}.
\end{remark}

\subsection{Off-Policy Semi-Model-Based Reinforcement Learning}
\label{subsec:AlgorithmDesign_3}

Next, we introduce the off-policy semi-model-based reinforcement learning approach. To facilitate the description of the algorithm, we first discuss an equivalent computational method corresponding to the policy evaluation step in Algorithm \ref{alg1:Nested_Iteration}.

\begin{lemma}
\label{alg3:lemma}
For any $Z^{(k,j+1)}$ and $K_2^{(k,j)}(j=0,1,\cdots)$ generated from \eqref{alg1:PolicyEvaluation} and \eqref{alg1:PolicyImprovement} in Algorithm \ref{alg1:Nested_Iteration} satisfy
\begin{equation}
\label{alg3:PolicyEvaluation}
\Phi^{(k,j)}_{i} \begin{bmatrix} \text{svec}(Z^{(k,j+1)}) \\ \text{vec}(S_2^{(k,j+1)}) \\ \text{svec}(D_{22}^{(k,j+1)}) \end{bmatrix} = \Theta^{(k,j)}_{i}  
\end{equation}
where 
\begin{equation*}
\begin{aligned}
&\Phi^{(k,j)}_{i} =\begin{bmatrix} \delta^{i,(k,0)}_{xx}& 2 \mathcal{I}^{i,(k,0)}_{xx} (I_{n} \otimes K_2^{(k,j)\top})-2\mathcal{I}^{i,(k,0)}_{xv_2} & \mathcal{I}^{i,(k,0)}_{xx}\bar{K}_2^{(k,j)} - \delta^{i,(k,0)}_{v_2v_2} \end{bmatrix},\\
&\Theta^{(k,j)}_{i} =  -\mathcal{I}^{i,(k,0)}_{xx}\text{vec}(M_{(k)}+ K_2^{(k,j)\top} R_{22,(k)} K_2^{(k,j)}),
\end{aligned}
\end{equation*}
and $Z^{(k,j+1)} \in \mathbb{S}^n$, $S_2^{(k,j+1)} \in \mathbb{R}^{m_2 \times n}$, $D_{22}^{(k,j+1)} \in \mathbb{S}^{m_2 \times m_2}$ denote the matrix variables at iteration $j$; their vectorized forms are $\text{svec}(Z^{(k,j+1)})$, $\text{vec}(S_2^{(k,j+1)})$, and $\text{svec}(D_{22}^{(k,j+1)})$, respectively.
\end{lemma}

\begin{proof}
This lemma is a direct corollary of Lemma 2 in \cite{Zhang2023}.    
\end{proof}

To ensure the uniqueness of the resulting matrix $Z^{(k,j+1)}, S_2^{(k,j+1)}, D_{22}^{(k,j+1)}$, we define the following data matrix.
For any positive integer $N$, we define
\begin{align*}
\Phi^{(k,j)} &= \left[\Phi^{(k,j)\top}_{1} , \dotsb, \Phi^{(k,j)\top}_{N} \right]^\top =\begin{bmatrix} \delta^{(k,0)}_{xx}& 2 \mathcal{I}^{(k,0)}_{xx} (I_{n} \otimes K_2^{(k,j)\top})-2\mathcal{I}^{(k,0)}_{xv_2} & \mathcal{I}^{(k,0)}_{xx}\bar{K}_2^{(k,j)} - \delta^{(k,0)}_{v_2v_2} \end{bmatrix},\\
\Theta^{(k,j)} &= \left[ \Theta^{(k,j)\top}_{1}, \dotsb, \Theta^{(k,j)\top}_{N} \right]^\top = -\mathcal{I}^{(k,0)}_{xx}\text{vec}(M_{(k)}+ K_2^{(k,j)\top} R_{22,(k)} K_2^{(k,j)}),
\end{align*}
where
\begin{align*}
\delta^{(k,0)}_{xx} &= \left[ \delta_{xx}^{1,(k,0)\top}, \dotsb, \delta_{xx}^{N,(k,0)\,\top} \right]^\top,\, \delta^{(k,0)}_{v_2v_2} = \left[ \delta_{v_2v_2}^{1,(k,0)\top}, \dotsb, \delta_{v_2v_2}^{N,(k,0)\top} \right]^\top, \\
\mathcal{I}^{(k,0)}_{xx} &= \left[ \mathcal{I}_{xx}^{1,(k,0)\,\top}, \dotsb, \mathcal{I}_{xx}^{N,(k,0)\,\top} \right]^\top,\, \mathcal{I}^{(k)}_{xv_2} = \left[ \mathcal{I}_{xv_2}^{1,(k,0)\,\top}, \dotsb, \mathcal{I}_{xv_2}^{N,(k,0)\,\top} \right]^\top.
\end{align*}

Then, the core iteration equation (for solving $Z^{(k,j+1)}, S_2^{(k,j+1)}, D_{22}^{(k,j+1)}$ ) is:
\[
\Phi^{(k,j)} \begin{bmatrix} \text{svec}(Z^{(k,j+1)}) \\ \text{vec}(S_2^{(k,j+1)}) \\ \text{svec}(D_{22}^{(k,j+1)}) \end{bmatrix} = \Theta^{(k,j)}.
\]
From Lemma 3 in \cite{Zhang2023}, if there is a $h \in \mathbb{Z}^+$, such that
\begin{equation}
\label{alg3:RankCondition}
\mathrm{rank}\bigl([\mathcal{I}^{(k,0)}_{xx},\ \mathcal{I}^{(k,0)}_{xv_2},\ \delta^{(k,0)}_{v_2v_2}]\bigr) = \frac{n(n+1)}{2}+m_2n+\frac{m_2(m_2+1)}{2},
\end{equation}
holds for any $i \geq h$, then $\Phi^{(k,j)}(\forall j \in \mathbb{Z})$ has full column rank. 
Under the full column rank condition of $\Phi^{(k,j)}$, the solution is unique:
\begin{equation}
\label{alg3:PolicyEvaluation+}
\begin{bmatrix} \text{svec}(Z^{(k,j+1)}) \\ \text{vec}(S_2^{(k,j+1)}) \\ \text{svec}(D_{22}^{(k,j+1)}) \end{bmatrix} = (\Phi^{(k,j)\top}\Phi^{(k,j)})^{-1}\Phi^{(k,j)\top}\Theta^{(k,j)}.
\end{equation}
\begin{remark}
For convenience of presentation, we assume throughout that \[K_2^{(k,j)}\in \mathcal{K}[A_{(k)},C_{1,(k)},\cdots,C_{r,(k)}; B_2,D_{1,2},\cdots, D_{r,2}]\,(j=0,1,2,\cdots)\] holds in Algorithm \ref{alg1:Nested_Iteration}. The corresponding guarantee is established in the proof of Lemma \ref{alg:lemma}.
\end{remark}

After elaborating on the symbol definitions and data matrix requirements, we present the following algorithm. 

\begin{algorithm}[htbp]
\caption{Off-Policy Semi-Model-Based Reinforcement Learning}
\label{alg3:SemiModelBased}
    \begin{algorithmic}[1]
        \State \textbf{Initialization}: 
        \State \quad Set iteration counters $k = 0$; convergence tolerance $\varepsilon > 0$. 
        \State \quad Initialize matrices $L \in \mathcal{A}$, $P^{(0)} = 0$, $R_{22,(0)}=R_{22}$, $M_{(0)}=Q-S^{\top}(0)R^{-1}(0)S(0)$, $\begin{bmatrix} K_{1}^{(0)} \\ K_{2}^{(0)} \end{bmatrix}=-R^{-1}(0)S(0)$.
        \Repeat
            \State \quad Reset inner counter $j = 0$; Initialize gain matrix $K_2^{(0,0)}=L+K_{2}^{(0)}$ or $K_2^{(k,0)}=K_2^{(0,0)}-K_2^{(k-1)}(k=1,2,\cdots)$.
            \State \quad Obtain $X^{(k,0)}(\cdot)$ and $v_2^{(k,0)}(\cdot)$ by running the SDE \eqref{zslqsdg:sde} with the
                $\begin{bmatrix} u_1^{(k,0)}(\cdot)\\ u_2^{(k,0)}(\cdot) \end{bmatrix}=\begin{bmatrix} K_{1}^{(0)}X(\cdot)\\ K_{2}^{(0)}X(\cdot)\end{bmatrix}+\begin{bmatrix} v_1^{(k,0)}(\cdot)\\ v_2^{(k,0)}(\cdot) \end{bmatrix}$ and $\begin{bmatrix} v_1^{(k,0)}(\cdot)\\ v_2^{(k,0)}(\cdot) \end{bmatrix}=\begin{bmatrix} K_{1}^{(k)}X(\cdot) \\ K_{2}^{(k)}X(\cdot)+\epsilon\end{bmatrix}$ on $[t, t+\Delta t]$. Compute $\delta_{xx}^{i,(k,0)},\delta_{v_2v_2}^{i,(k,0)},\mathcal{I}_{xx}^{i,(k,0)},\mathcal{I}_{xv_2}^{i,(k,0)}(i=1,2,\cdots,N)$.
                \Repeat
                    \State \textbf{Policy Evaluation}:
                    \State Solve $Z^{(k,j+1)}, S_2^{(k,j+1)}, D_{22}^{(k,j+1)}$ via \eqref{alg3:PolicyEvaluation+} by $\Phi^{(k,j)}$ and $\Theta^{(k,j)}$.
                    \State \textbf{Policy Improvement}: Update the gain matrix $K_2^{(k,j)}$ to $K_2^{(k,j+1)}$:
                    \begin{align*}
                        K_2^{(k,j+1)} = - (R_{22,(k)} + D_{22}^{(k,j+1)})^{-1}S_2^{(k,j+1)}.
                    \end{align*}
                    \State \quad Increment inner counter: $j \gets j + 1$.
                \Until{$\| Z^{(k,j)} - Z^{(k,j-1)} \| < \varepsilon$}
            \State \quad Update $Z^{(k)} =Z^{(k,j)},\,P^{(k+1)}= P^{(k)} + Z^{(k)}$. 
            \State \quad Evolve matrices $S_{(k+1)}, N_{1,(k)}, N_{2,(k)}, R_{(k+1)},R_{12,(k+1)}, R_{22,(k+1)}$, and $R^{\sharp}_{22,(k+1)}$ via \eqref{iterative:matrices_evolve} by using $P^{(k+1)}$ and $Z^{(k)}$. 
            \State \quad Update $\begin{bmatrix} K_{1}^{(k+1)} \\ K_{2}^{(k+1)} \end{bmatrix} =-R_{(k+1)}^{-1}S_{(k+1)} $, $C_{l,(k+1)} = C_{l} +  \begin{bmatrix} D_{l,1} & D_{l,2} \end{bmatrix} \begin{bmatrix} K_{1}^{(k+1)} \\ K_{2}^{(k+1)} \end{bmatrix}(l=1,2,\cdots,r)$, and $M_{(k+1)} =-\bigl(N_{1,(k)}-R_{12,(k+1)}R_{22,(k+1)}^{-1}N_{2,(k)}\bigr)^{\top} R^{\sharp-1}_{22,(k+1)}\bigl(N_{1,(k)}-R_{12,(k+1)}R_{22,(k+1)}^{-1}N_{2,(k)}\bigr)$.
            \State \quad Increment outer counter: $k \gets k + 1$.
        \Until{$\| P^{(k)} - P^{(k-1)} \| < \varepsilon$} 
        \State \Return $P^{(k)}$
    \end{algorithmic}
\end{algorithm}

\begin{remark}
Algorithms \ref{alg2:SemiModelBased} and \ref{alg3:SemiModelBased} require additional information beyond the initial system parameter $A$. The key distinction between the two algorithms resides in the disparate sampling frequencies they demand. Specifically, Algorithm \ref{alg2:SemiModelBased} necessitates a complete round of sampling per iteration, whereas Algorithm \ref{alg3:SemiModelBased} only requires a single sampling operation at the outset of each inner-loop iteration. The implementation mechanisms of these two algorithms are expected to provide valuable insights for the subsequent design of new algorithms addressing this problem.
\end{remark}

\clearpage

\subsection{Model-Free Reinforcement Learning}
\label{subsec:AlgorithmDesign_4}

Next, we present a fully model-free reinforcement learning algorithm. Unlike Algorithm \ref{alg2:SemiModelBased} and Algorithm \ref{alg3:SemiModelBased},  Algorithm \ref{alg4:ModelFree} does not require prior knowledge of system parameters. In this algorithm, we solve the problem offline using a single batch of data collected at the beginning. Different from the sampling matrix \ref{alg23:CollectedData} described earlier, we redefine the sampling matrix as follows:
\begin{equation*}
\label{alg4:CollectedData}
\begin{aligned}
&\delta_{xx}^{i} := \mathbb{E}^{\mathcal{F}_{t_i}} \left[ \bar{X}^{\top}(t_{i+1}) - \bar{X}^{\top}(t_i) \right],\delta_{v_1v_1}^i := \mathbb{E}^{\mathcal{F}_{t_i}} \left[ \int_{t_i}^{t_{i+1}} \bar{v}_1^{\top}(s)ds \right],\delta_{v_2v_2}^i := \mathbb{E}^{\mathcal{F}_{t_i}} \left[ \int_{t_i}^{t_{i+1}} \bar{v}_2^{\top}(s)ds \right],\\
&\mathcal{I}_{xx}^i := \mathbb{E}^{\mathcal{F}_{t_i}} \left[ \int_{t_i}^{t_{i+1}} X^{\top}(s) \otimes X^{\top}(s)ds \right],\mathcal{I}_{v_2v_1}^i := \mathbb{E}^{\mathcal{F}_{t_i}} \left[ \int_{t_i}^{t_{i+1}} v_2^{\top}(s) \otimes v_1^{\top}(s)ds \right],\\
&\mathcal{I}_{xv_1}^i := \mathbb{E}^{\mathcal{F}_{t_i}} \left[ \int_{t_i}^{t_{i+1}} X^{\top}(s) \otimes v_1^{\top}(s)ds \right],\mathcal{I}_{xv_2}^i := \mathbb{E}^{\mathcal{F}_{t_i}} \left[ \int_{t_i}^{t_{i+1}} X^{\top}(s) \otimes v_2^{\top}(s)ds \right],\\
&\mathcal{I}_{v_1v_1}^i := \mathbb{E}^{\mathcal{F}_{t_i}} \left[ \int_{t_i}^{t_{i+1}} v_1^{\top}(s) \otimes v_1^{\top}(s)ds \right],\mathcal{I}_{v_2v_2}^i := \mathbb{E}^{\mathcal{F}_{t_i}} \left[ \int_{t_i}^{t_{i+1}} v_2^{\top}(s) \otimes v_2^{\top}(s)ds \right].
\end{aligned}
\end{equation*}

We first discuss the mechanism for policy evaluation in Algorithm \ref{alg4:ModelFree} via the following lemma.

\begin{lemma}
\label{alg4:lemma1}
For any $Z^{(k,j+1)}$ ($k=0,1,2,\cdots;j=0,1,2,\cdots$) generated from
\begin{equation*}
\label{alg4:lemma_Lyapunov}
\begin{aligned}
&Z^{(k,j+1)}(A_{(k)} + B_2 K^{(k,j)}_2) + (A_{(k)} + B_2 K^{(k,j)}_2)^\top Z^{(k,j+1)} \\
&+ \sum_{l=1}^r (C_{l,(k)} + D_{l,2} K^{(k,j)}_2)^\top Z^{(k,j+1)} (C_{l,(k)} + D_{l,2} K^{(k,j)}_2)  \\
&+ M_{(k)} + K_2^{(k,j)\top} R_{22,(k)} K_2^{(k,j)} = 0
\end{aligned}
\end{equation*}
satisfy
\begin{equation*}
\label{alg4:PolicyEvaluation}
\Phi^{(k,j)}_{i} \begin{bmatrix} \text{svec}(Z^{(k,j+1)}) \\ \text{vec}(S_1^{(k,j+1)}) \\ \text{vec}(S_2^{(k,j+1)}) \\ \text{svec}(D_{11}^{(k,j+1)}) \\\text{svec}(D_{22}^{(k,j+1)})\\\text{vec}(D_{12}^{(k,j+1)})\end{bmatrix} = \Theta^{(k,j)}_{i},  
\end{equation*}
where 
\begin{equation*}
\label{alg4:lemma1_eq1}
\begin{aligned}
&A_{(k)}=A + B_{1} K_1^{(k)} + B_{2} K_2^{(k)},\,\,C_{l,(k)}=C_{l} + D_{l,1} K_1^{(k)} + D_{l,2} K_2^{(k)}(l=1,2,\cdots),\\
&K_2^{(k,j)} = - ( R_{22,(k)} + \sum_{l=1}^r D_{l,2}^\top Z^{(k,j)} D_{l,2} )^{-1} (B_2^\top Z^{(k,j)}  + \sum_{l=1}^r D_{l,2}^\top Z^{(k,j)} C_{l,(k)}),\\
&S_1^{(k,j+1)}=B_1^{\top} Z^{(k,j+1)} + \sum_{l=1}^{r}D_{l,1}^{\top} Z^{(k,j+1)}C_{l,(k)},S_2^{(k,j+1)}=B_2^{\top} Z^{(k,j+1)} + \sum_{l=1}^{r}D_{l,2}^{\top} Z^{(k,j+1)}C_{l,(k)},\\
&D_{11}^{(k,j+1)}=\sum_{l=1}^{r}D_{l,1}^{\top} Z^{(k,j+1)} D_{l,1},\,D_{22}^{(k,j+1)}=\sum_{l=1}^{r}D_{l,2}^{\top} Z^{(k,j+1)} D_{l,2},\,D_{12}^{(k,j+1)}=\sum_{l=1}^{r}D_{l,1}^{\top} Z^{(k,j+1)} D_{l,2}
\end{aligned}
\end{equation*}
\begin{equation*}
\label{alg4:lemma1_eq2}
\begin{aligned}
&\Phi^{(k,j)}_i =\begin{bmatrix} \delta^{i}_{xx}& H^{i,(k,j)}_{1} & H^{i,(k,j)}_{2} & H^{i,(k,j)}_{3} & H^{i,(k,j)}_{4} & H^{i,(k,j)}_{5}\end{bmatrix},\\
&\Theta^{(k,j)}_{i} =  -\mathcal{I}^{i}_{xx}\text{vec}(M_{(k)}+ K_2^{(k,j)\top} R_{22,(k)} K_2^{(k,j)}),
\end{aligned}
\end{equation*}
with 
\begin{equation*}
\label{alg4:lemma1_eq3}
\begin{aligned}
&H^{i,(k,j)}_{1}=2 \mathcal{I}^{i}_{xx} (I_{n} \otimes K_1^{(k)\top})-2\mathcal{I}^{i}_{xv_1},\\
&H^{i,(k,j)}_{2}=2 \mathcal{I}^{i}_{xx} (I_{n} \otimes (K_2^{(k)}+K_2^{(k,j)})^{\top})-2\mathcal{I}^{i}_{xv_2},\\
&H^{i,(k,j)}_{3}=2\mathcal{I}^{i}_{xv_1}(K_1^{(k)\top} \otimes I_{m_1})\mathcal{T}-\delta_{v_1v_1}^i-\mathcal{I}^{i}_{xx}\bar{K}_1^{(k)},\\
&H^{i,(k,j)}_{4}=2\mathcal{I}^{i}_{xv_2}(K_2^{(k)\top} \otimes I_{m_2})\mathcal{T}-\delta_{v_2v_2}^i-\mathcal{I}^{i}_{xx}\bar{K}_2^{(k)}+\mathcal{I}^{i}_{xx}\bar{K}_2^{(k,j)},\\
&H^{i,(k,j)}_{5}=2\mathcal{I}^{i}_{xv_1}(K_2^{(k)\top} \otimes I_{m_1} )+ 2\mathcal{I}^{i}_{xv_2}(K_1^{(k)\top} \otimes I_{m_2})-2\mathcal{I}^{i}_{v_2v_1}-2\mathcal{I}^{i}_{xx}(K_2^{(k)\top}\otimes K_1^{(k)\top}),
\end{aligned}
\end{equation*}
and $Z^{(k,j+1)} \in \mathbb{S}^n$, $S_1^{(k,j+1)}\in \mathbb{R}^{m_1 \times n}$, $S_2^{(k,j+1)}\in \mathbb{R}^{m_2 \times n}$ $D_{11}^{(k,j+1)} \in \mathbb{S}^{m_1 \times m_1}$, $D_{22}^{(k,j+1)} \in \mathbb{S}^{m_2 \times m_2}$ and  $D_{12}^{(k,j+1)} \in \mathbb{R}^{m_1 \times m_2}$.
\end{lemma}

\begin{proof}
First, we write the SDE \eqref{zslqsdg:sde} as:
\begin{align*}
dX(s) =&\bigg[ \left( A + B_1 K_1^{(k)} + B_2 K_2^{(k)} + B_2 K_2^{(k,j)}\right) X(s)  + B_1 \left( v_1(s) - K_1^{(k)} X(s) \right)\\
&+ B_2 \left( v_2(s) - K_2^{(k)} X(s) -K_2^{(k,j)}X(s)\right) \bigg] ds\\
+&\sum^{r}_{l=1} \bigg[ \left( C_{l} + D_{l,1} K_1^{(k)} + D_{l,2} K_2^{(k)}+D_{l,2} K_2^{(k,j)} \right) X(s)  + D_{l,1} \left( v_1(s) -  K_1^{(k)} X(s) \right) \\
&+D_{l,2} \left( v_2(s)- K_2^{(k)} X(s) -K_2^{(k,j)}X(s) \right)\bigg] dW(s). 
\end{align*}

Applying Itô's formula to $X^{\top}(s) Z^{(k,j+1)} X(s)$, we obtain:
\begin{align*}
&d\left( X^{\top}(s) Z^{(k,j+1)} X(s) \right) = \bigg[\left( v_1(s) - K_1^{(k)} X(s) \right)^{\top}S_1^{(k,j)}X(s)+X(s)^{\top}S_1^{(k,j)\top}\left( v_1(s) - K_1^{(k)} X(s) \right)+\\
&\left( v_2(s) - K_2^{(k)} X(s) -K_2^{(k,j)}X(s) \right)^{\top}S_2^{(k,j)}X(s)+X(s)^{\top}S_2^{(k,j)\top}\left( v_2(s) - K_2^{(k)} X(s) -K_2^{(k,j)}X(s) \right)+\\
&\left( v_1(s) -  K_1^{(k)} X(s) \right)^{\top}D_{11}^{(k,j+1)}\left( v_1(s) -  K_1^{(k)} X(s) \right)+\\
&\left( v_2(s) -  K_2^{(k)} X(s) -K_2^{(k,j)}X(s)\right)^{\top}D_{22}^{(k,j+1)}\left( v_2(s) -  K_2^{(k)} X(s) +K_2^{(k,j)}X(s)\right)+\\
& \left( v_1(s) -  K_1^{(k)} X(s) \right)^{\top}D_{12}^{(k,j+1)}\left( v_2(s) -  K_2^{(k)} X(s)\right)+\left( v_2(s) -  K_2^{(k)} X(s)\right)^{\top}D_{21}^{(k,j+1)}\left( v_1(s) -  K_1^{(k)} X(s) \right)+\\
&X^{\top}(s) \Big(Z^{(k,j+1)}\big( A + B_1 K_1^{(k)} + B_2 K_2^{(k)} + B_2 K_2^{(k,j)} \big)+\big( A + B_1 K_1^{(k)} + B_2 K_2^{(k)} + B_2 K_2^{(k,j)} \big)^{\top} Z^{(k,j+1)}\Big) X(s)+\\
&X^{\top}(s)\sum^{r}_{l=1}\left( C_{l} + D_{l,1} K_1^{(k)} + D_{l,2} K_2^{(k)}+D_{l,2} K_2^{(k,j)} \right) ^{\top}Z^{(k,j+1)}\left( C_{l} + D_{l,1} K_1^{(k)} + D_{l,2} K_2^{(k)}+D_{l,2} K_2^{(k,j)} \right)X(s) \bigg] dt\\
&+ (\cdots) dW(s).
\end{align*}
Integrating both sides of equation from $t_{i}$ to $t_{i+1}$ and taking the conditional expectation $\mathbb{E}^{\mathcal{F}_{t_i}}$, we can obtain:
\begin{align*}
&\mathbb{E}^{\mathcal{F}_{t_i}} \left[ X^{\top}(t+\Delta t) Z^{(k,j+1)} X(t+\Delta t) - X^{\top}(t) Z^{(k,j+1)} X(t) \right] \\
=&\mathbb{E}^{\mathcal{F}_{t_i}} \left[ \int_{t_i}^{t_{i+1}} \Bigg\{ 2\left( v_1(s) - K_1^{(k)} X(s) \right)^{\top}S_1^{(k,j+1)}X(s)+2\left( v_2(s) - K_2^{(k)} X(s) -K_2^{(k,j)}X(s) \right)^{\top}S_2^{(k,j+1)}X(s)\right. \\
&\left.+\left( v_1(s) -  K_1^{(k)} X(s) \right)^{\top}D_{11}^{(k,j+1)}\left( v_1(s) -  K_1^{(k)} X(s) \right)  \right. \\
&\left.+\left( v_2(s) -  K_2^{(k)} X(s) -K_2^{(k,j)}X(s)\right)^{\top}D_{22}^{(k,j+1)}\left( v_2(s) -  K_2^{(k)} X(s) +K_2^{(k,j)}X(s)\right) \right.\\
& \left. + 2\left( v_1(s) -  K_1^{(k)} X(s) \right)^{\top}D_{12}^{(k,j+1)}\left( v_2(s) -  K_2^{(k)} X(s) \right) - X^{\top}(s) \left( M_{(k)}+ K_2^{(k,j)\top} R_{22,(k)} K_2^{(k,j)}\right) X(s) \Bigg\} ds \right]. 
\end{align*}
Using vectorization technology and the Kronecker product $\otimes$, we write it as:
\begin{equation}
\label{alg4:lemma1_proof1}
\begin{aligned}
&-\mathbb{E}^{\mathcal{F}_{t_i}} \left[ \int_{t_i}^{t_{i+1}} (X^{\top}(s) \otimes X^{\top}(s)) \mathrm{vec}\left( M_{(k)}+ K_2^{(k,j)\top} R_{22,(k)} K_2^{(k,j)}\right) ds \right] \\
=&\mathbb{E}^{\mathcal{F}_{t_i}} \Bigg[ \bar{X}^{\top}(t_{i+1}) \mathrm{svec}\left( Z^{(k,j+1)} \right) - \bar{X}^{\top}(t_i) \mathrm{svec}\left( Z^{(k,j+1)} \right) \\
&- \int_{t_i}^{t_{i+1}} \left( 2X^{\top}(s) \otimes \left( v_1(s) - K_1^{(k)} X(s) \right)^{\top}\right) \mathrm{vec}\left( S_1^{(k,j+1)} \right) ds \\
&- \int_{t_i}^{t_{i+1}} \left( 2 X^{\top}(s) \otimes \left(v_2(s) - K_2^{(k)} X(s) -K_2^{(k,j)}X(s) \right)^{\top} \right) \mathrm{vec}\left( S_2^{(k,j+1)} \right) ds \\
&- \int_{t_i}^{t_{i+1}} \left( \left( v_1(s) -  K_1^{(k)} X(s)\right)^{\top} \otimes \left( v_1(s) -  K_1^{(k)} X(s)\right)^{\top} \mathcal{T}\mathrm{svec}\left( D_{11}^{(k,j+1)} \right)\right) ds\\
&- \int_{t_i}^{t_{i+1}} \left( \left( v_2(s) -  K_2^{(k)} X(s) +K_2^{(k,j)}X(s) \right)^{\top} \otimes \left( v_2(s) -  K_2^{(k)} X(s) -K_2^{(k,j)}X(s) \right)^{\top} \right)\mathcal{T}\mathrm{svec}\left( D_{22}^{(k,j+1)} \right) ds \\
&- \int_{t_i}^{t_{i+1}} \left(2\left( v_2(s) -  K_2^{(k)} X(s) \right)^{\top} \otimes \left( v_1(s) -  K_1^{(k)} X(s)\right)^{\top} \right)\mathrm{vec}\left( D_{12}^{(k,j+1)} \right) ds  \Bigg].
\end{aligned}
\end{equation}
Based on the notations defined earlier, \eqref{alg4:PolicyEvaluation} can be derived from the \eqref{alg4:lemma1_proof1}.
\end{proof}

To ensure the uniqueness of the resulting matrix, we define the following data matrix.
For any positive integer $N$, we define
\begin{equation*}
\begin{aligned}
\Phi^{(k,j)} &= \left[\Phi^{(k,j)\top}_{1} , \dotsb, \Phi^{(k,j)\top}_{N} \right]^\top \\
&\,=\begin{bmatrix} \delta_{xx},\, H^{(k,j)}_{1},\, H^{(k,j)}_{2} ,\,H^{(k,j)}_{3},\,H^{(k,j)}_{4},\, H^{(k,j)}_{5}\end{bmatrix},\\
\Theta^{(k,j)} &= \left[ \Theta^{(k,j)\top}_{1}, \dotsb, \Theta^{(k,j)\top}_{N} \right]^\top= -\mathcal{I}_{xx}\text{vec}(M_{(k)}+ K_2^{(k,j)\top} R_{22,(k)} K_2^{(k,j)}),
\end{aligned}
\end{equation*}
where
\begin{align*}
\delta_{xx} = \left[ \delta_{xx}^{1\top}, \dotsb, \delta_{xx}^{N\,\top} \right]^\top,\mathcal{I}_{xx} = \left[ \mathcal{I}_{xx}^{1\,\top}, \dotsb, \mathcal{I}_{xx}^{N\,\top} \right]^\top, \\
\,\, H^{(k,j)}_{s}= \left[ H^{1,(k,j)^\top}_{s}, \dotsb, H^{N,(k,j)^\top}_{s} \right]^\top(s=1,2,\cdots,5).
\end{align*}

The core iteration equation (for solving $Z^{(k,j+1)}$, $S_1^{(k,j+1)}$, $S_2^{(k,j+1)}$, $D_{11}^{(k,j+1)}$, $D_{22}^{(k,j+1)} $ and $D_{12}^{(k,j+1)}$ ) is:
\[
\Phi^{(k,j)}\begin{bmatrix} \text{svec}(Z^{(k,j+1)}) \\ \text{vec}(S_1^{(k,j+1)}) \\ \text{vec}(S_2^{(k,j+1)}) \\ \text{svec}(D_{11}^{(k,j+1)}) \\\text{svec}(D_{22}^{(k,j+1)})\\\text{vec}(D_{12}^{(k,j+1)})\end{bmatrix} = \Theta^{(k,j)}.
\]
Next, we present a rank condition to ensure that the solution is unique. In this case, 
\begin{equation}
\label{alg4:PolicyEvaluation+}
\begin{bmatrix} \text{svec}(Z^{(k,j+1)}) \\ \text{vec}(S_1^{(k,j+1)}) \\ \text{vec}(S_2^{(k,j+1)}) \\ \text{svec}(D_{11}^{(k,j+1)}) \\\text{svec}(D_{22}^{(k,j+1)})\\\text{vec}(D_{12}^{(k,j+1)})\end{bmatrix} = (\Phi^{(k,j)\top}\Phi^{(k,j)})^{-1}\Phi^{(k,j)\top}\Theta^{(k,j)}.
\end{equation}

\begin{lemma}
\label{alg4:lemma2}
If the feedback $(K_1^{(k)}, K_2^{(k)}+K_2^{(k,j)}) \in \mathcal{K}$ and there exists a $h \in \mathbb{Z^{+}}$, such that, for all $i \geq h$,
\begin{equation}
\label{alg:rank}
rank([\delta_{xx},\,\mathcal{I}_{xv_1},\,\mathcal{I}_{xv_2},\,\delta_{v_1v_1},\,\delta_{v_2v_2},\,\mathcal{I}_{v_2v_1}])=\frac{n(n+1)}{2}+m_1n+m_2n+\frac{m_1(m_1+1)}{2}+\frac{m_2(m_2+1)}{2}+m_1m_2,
\end{equation}
where
\begin{align*}
&\delta_{xx} = \left[ \delta_{xx}^{1\,\top}, \dotsb, \delta_{xx}^{N\,\top} \right]^\top,\,\, \delta_{v_1v_1} = \left[ \delta_{v_1v_1}^{1\,\top}, \dotsb, \delta_{v_1v_1}^{N\,\top} \right]^\top,\,\, \delta_{v_2v_2} = \left[ \delta_{v_2v_2}^{1\,\top}, \dotsb, \delta_{v_2v_2}^{N\,\top} \right]^\top,\\\\
&\mathcal{I}_{xv_1} = \left[ \mathcal{I}_{xv_1}^{1\,\top}, \dotsb, \mathcal{I}_{xv_1}^{N\,\top} \right]^\top,\,\, \mathcal{I}_{xv_2} = \left[ \mathcal{I}_{xv_2}^{1\,\top}, \dotsb, \mathcal{I}_{xv_2}^{N\,\top} \right]^\top,\,\, \mathcal{I}_{v_2v_1} = \left[ \mathcal{I}_{v_2v_1}^{1\,\top}, \dotsb, \mathcal{I}_{v_2v_1}^{N\,\top} \right]^\top,\\
\end{align*}
then, $\Phi^{(k,j)}$ has full column rank.    
\end{lemma}

\begin{proof}
To prove that $\Phi^{(k,j)}$ has full column rank, it is equivalent to show that the equation  
\begin{equation}
\label{alg4:lemma2_proof1}    
\Phi^{(k,j)}V = 0,V \in \mathbb{R}^{d},d = \frac{n(n+1)}{2}+m_1n+m_2n+\frac{m_1(m_1+1)}{2}+\frac{m_2(m_2+1)}{2}+m_1m_2
\end{equation}
has the unique solution $V = 0$.
To achieve it, we now prove it by contradiction.

We assume
\[
V=[\text{svec}(U)^{\top}, \text{vec}(W_1)^{\top}, \text{vec}(W_2)^{\top}, \text{svec}(D_{11})^{\top}, \text{svec}(D_{22})^{\top}, \text{vec}(D_{12})^{\top}]^{\top} \in  \mathbb{R}^{d}
\] 
is a nonzero column vector, where 
\[
\text{svec}(U) \in \mathbb{R}^{\frac{n(n+1)}{2}}, \text{vec}(W_1)\in \mathbb{R}^{m_1n}, \text{vec}(W_2)\in \mathbb{R}^{m_2n}, \text{svec}(D_{11})\in \mathbb{R}^\frac{m_1(m_1+1)}{2}, \text{svec}(D_{22})\in  \mathbb{R}^\frac{m_2(m_2+1)}{2}.
\]
Then, by the definitions of $\text{svec}(\cdot)$ and $\text{vec}(\cdot)$, symmetric matrices $U\in\mathbb{S}^{n}$, and matrix $W_1\in\mathbb{R}^{m_1\times n},W_2\in\mathbb{R}^{m_2\times n},D_{11}\in\mathbb{S}^{m_1\times m_1},D_{22}\in\mathbb{S}^{m_2\times m_2},D_{12}\in\mathbb{R}^{m_1\times m_2}$ can be uniquely determined by $\text{svec}(U), \text{vec}(W_1), \text{vec}(W_2),\\ \text{svec}(D_{11}), \text{svec}(D_{22}), \text{vec}(D_{12})$, respectively.

Applying Itô's formula to $X(s)^{\top} U X(s)$, we obtain:
\begin{equation*}
\begin{aligned}
&\mathbb{E}\big[X(t+\Delta t)^{\top}UX(t+\Delta t)-X(t)^{\top}UX(t)\big]\\
=&\mathbb{E}\int_{t}^{t+\Delta t}X(s)^{\top}\big(A^{\top}U+UA+C^{\top}UC\big)X(s)ds\\
&+2\mathbb{E}\int_{t}^{t+\Delta t}v_1(s)^{\top}(B_1^{\top}U+\sum^r_{l=1}D_{l,1}^{\top}UC)X(s)ds+2\mathbb{E}\int_{t}^{t+\Delta t}v_2(s)^{\top}(B_2^{\top}U+\sum^r_{l=1}D_{l,2}^{\top}UC)X(s)ds\\
&+\mathbb{E}\int_{t}^{t+\Delta t}v_1(s)^{\top}\sum^r_{l=1}D_{l,1}^{\top}UD_{l,1}v_1(s)ds+\mathbb{E}\int_{t}^{t+\Delta t}v_2(s)^{\top}\sum^r_{l=1}D_{l,2}^{\top}UD_{l,2}v_2(s)ds\\
&+2\mathbb{E}\int_{t}^{t+\Delta t}v_1(s)^{\top}\sum^r_{l=1}D_{l,1}^{\top}UD_{l,2}v_2(s)ds.
\end{aligned}
\end{equation*}
Using vectorization technology and the Kronecker product $\otimes$, we write it as:
\begin{equation}
\label{alg4:lemma2_proof2}
\begin{aligned}
&\delta_{xx}\text{svec}(U)=\mathcal{I}_{xx}\text{vec}(A^{\top}U+UA+C^{\top}UC)\\
&+2\mathcal{I}_{xv_1}\text{vec}(B_1^{\top}U+\sum^r_{l=1}D_{l,1}^{\top}UC)+2\mathcal{I}_{xv_2}\text{vec}(B_2^{\top}U+\sum^r_{l=1}D_{l,2}^{\top}UC)\\
&+\delta_{v_1v_1}\text{svec}(\sum^r_{l=1}D_{l,1}^{\top}UD_{l,1})+\delta_{v_2v_2}\text{svec}(\sum^r_{l=1}D_{l,2}^{\top}UD_{l,2})+2\mathcal{I}_{v_2v_1}\text{vec}(\sum^r_{l=1}D_{l,1}^{\top}UD_{l,2}).
\end{aligned}
\end{equation}

Using Lemma \ref{alg4:lemma1}, \eqref{alg4:lemma2_proof2} and the definition of $\Phi^{(k,j)}$, we have
\begin{equation*} 
\begin{aligned}
\Phi^{(k,j)}V =& \delta_{xx}\text{svec}(\mathbb{U})+2\mathcal{I}_{xv_1}\text{vec}(\mathbb{W}_1)+2\mathcal{I}_{xv_2}\text{vec}(\mathbb{W}_2)\\
&+\delta_{v_1v_1}\text{svec}(\mathbb{D}_{11})+\delta_{v_2v_2}\text{svec}(\mathbb{D}_{22})+\mathcal{I}_{v_2v_1}\text{vec}(\mathbb{D}_{12})
\end{aligned}
\end{equation*}
where
\begin{equation}
\label{alg4:lemma2_proof3}
\begin{aligned}
&\mathbb{U}=A^{\top}U+UA+C^{\top}UC+K_1^{(k)\top}W_1+W_1^{\top}K_1^{(k)}-K_1^{(k)\top}D_{11}K_1^{(k)}\\
&+(K_2^{(k)}+K_2^{(k,j)})^{\top}W_2+W_2^{\top}(K_2^{(k)}+K_2^{(k,j)})-K_2^{(k)\top}D_{22}K_2^{(k)}+K_2^{(k,j)\top}D_{22}K_2^{(k,j)}\\
&-K_1^{(k)\top}D_{12}K_2^{(k)}-K_2^{(k)\top}D_{21}K_1^{(k)},\\
&\mathbb{W}_1=B_1^{\top}U+\sum^r_{l=1}D_{l,1}^{\top}UC+D_{11}K_1^{(k)}+D_{12}K_2^{(k)}-W_1,\\
&\mathbb{W}_2=B_2^{\top}U+\sum^r_{l=1}D_{l,2}^{\top}UC+D_{22}K_2^{(k)}+D_{12}^{\top}K_1^{(k)}-W_2,\\
&\mathbb{D}_{11}=\sum^r_{l=1}D_{l,1}^{\top}UD_{l,1}-D_{11}, \mathbb{D}_{22}=\sum^r_{l=1}D_{l,2}^{\top}UD_{l,2}-D_{22}, \mathbb{D}_{12}=\sum^r_{l=1}D_{l,1}^{\top}UD_{l,2}-D_{12}.
\end{aligned}
\end{equation}

Then, we have
\begin{equation}
\label{alg4:lemma2_proof4}
[\delta_{xx},\,\mathcal{I}_{xv_1},\,\mathcal{I}_{xv_2},\,\delta_{v_1v_1},\,\delta_{v_2v_2},\,\mathcal{I}_{v_2v_1}] \begin{bmatrix} \text{svec}(\mathbb{U}) \\ \text{vec}(\mathbb{W}_1) \\ \text{vec}(\mathbb{W}_2) \\ \text{svec}(\mathbb{D}_{11}) \\\text{svec}(\mathbb{D}_{22})\\\text{vec}(\mathbb{D}_{12})\end{bmatrix} =  0.
\end{equation}

It is easy to see that $[\delta_{xx},\,\mathcal{I}_{xv_1},\,\mathcal{I}_{xv_2},\,\delta_{v_1v_1},\,\delta_{v_2v_2},\,\mathcal{I}_{v_2v_1}]$ has full column rank under condition. Then, the solution to \eqref{alg4:lemma2_proof4} is 
\[
\text{svec}(\mathbb{U})=0,\text{vec}(\mathbb{W}_1)=0,\text{vec}(\mathbb{W}_2)=0,\text{svec}(\mathbb{D}_{11})=0,\text{svec}(\mathbb{D}_{22})=0,\text{vec}(\mathbb{D}_{12})=0.
\]

From \eqref{alg4:lemma2_proof3}, we have
\begin{equation}
\label{alg4:lemma2_proof5}
\begin{aligned}
0&=A^{\top}U+UA+C^{\top}UC+K_1^{(k)\top}W_1+W_1^{\top}K_1^{(k)}-K_1^{(k)\top}D_{11}K_1^{(k)}\\
&+(K_2^{(k)}+K_2^{(k,j)})^{\top}W_2+W_2^{\top}(K_2^{(k)}+K_2^{(k,j)})-K_2^{(k)\top}D_{22}K_2^{(k)}+K_2^{(k,j)\top}D_{22}K_2^{(k,j)},\\
&-K_1^{(k)\top}D_{12}K_2^{(k)}-K_2^{(k)\top}D_{12}^{\top}K_1^{(k)},
\end{aligned}
\end{equation}
where
\begin{align*}
&D_{11}=\sum^r_{l=1}D_{l,1}^{\top}UD_{l,1},D_{22}=\sum^r_{l=1}D_{l,2}^{\top}UD_{l,2},D_{12}=\sum^r_{l=1}D_{l,1}^{\top}UD_{l,2},\\
&W_1=B_1^{\top}U+\sum^r_{l=1}D_{l,1}^{\top}UC+\sum^r_{l=1}D_{l,1}^{\top}UD_{l,1}K_1^{(k)}+\sum^r_{l=1}D_{l,1}^{\top}UD_{l,2}K_2^{(k)},\\
&W_2=B_2^{\top}U+\sum^r_{l=1}D_{l,2}^{\top}UC+\sum^r_{l=1}D_{l,2}^{\top}UD_{l,2}K_2^{(k)}+\sum^r_{l=1}D_{l,2}^{\top}UD_{l,1}K_1^{(k)}.
\end{align*}

By transforming the \eqref{alg4:lemma2_proof5}, we obtain
\begin{equation*}
\label{alg4:lemma2_proof6}
\begin{aligned}
0&=(A+B_1K_1^{(k)}+B_2K_2^{(k)}+B_2K_2^{(k,j)})^{\top}U+U(A+B_1K_1^{(k)}+B_2K_2^{(k)}+B_2K_2^{(k,j)})\\
&+\sum^r_{l=1}(C+D_{l,1}K_1^{(k)}+D_{l,2}K_2^{(k)}+D_{l,2}K_2^{(k,j)})^{\top}U(C+D_{l,1}K_1^{(k)}+D_{l,2}K_2^{(k)}+D_{l,2}K_2^{(k,j)})
\end{aligned}
\end{equation*}

Since the feedback $(K_1^{(k)}, K_2^{(k)}+K_2^{(k,j)}) \in \mathcal{K}$, we have $U=0$.
Based on this, we obtain $V = 0$, which contradicts the assumption that \eqref{alg4:lemma2_proof1} has a non-trivial solution. Therefore, when the rank condition is satisfied, $\Phi^{(k,j)}$ has full column rank.
\end{proof}

We now present the description of the model-free reinforcement learning algorithm.

\begin{algorithm}[htbp]
\caption{Model-Free Reinforcement Learning}
\label{alg4:ModelFree}
    \begin{algorithmic}[1]
        \State \textbf{Initialization}: 
        \State \quad Set iteration counters $k = 0$; convergence tolerance $\varepsilon > 0$. 
        \State \quad Initialize matrices $L \in \mathcal{A}$, $P^{(0)} = 0$, $R_{22,(0)}=R_{22}$, $\begin{bmatrix} K_{1}^{(0)} \\ K_{2}^{(0)} \end{bmatrix}=-R^{-1}(0)S(0)$, $M_{(0)}=Q-S^{\top}(0) R^{-1}(0)S(0)$.
        \State \quad Set $\begin{bmatrix} u_1(\cdot)\\ u_2(\cdot) \end{bmatrix}=\begin{bmatrix} K_{1}^{(0)}X(\cdot)\\ K_{2}^{(0)}X(\cdot)\end{bmatrix}+\begin{bmatrix} v_1(\cdot)\\ v_2(\cdot) \end{bmatrix}$ and obtain local state trajectory $X(s),v_1(s),v_2(s)$ by running the system on $[t, t+\Delta t]$. Compute $\delta_{xx},\,\mathcal{I}_{xv_1},\,\mathcal{I}_{xv_2},\,\delta_{v_1v_1},\,\delta_{v_2v_2},\,\delta_{v_2v_1}$.
        \Repeat
            \State \quad Reset inner counter $j = 0$; Initialize gain matrix $K_2^{(0,0)}=L$ or $K_2^{(k,0)}=K_2^{(0,0)}-K_2^{(k-1)}(k=1,2,\cdots)$.
                \Repeat
                    \State \textbf{Policy Evaluation}:
                    \State Compute $\Phi^{(k,j)}$ and $\Theta^{(k,j)}$. Solve $Z^{(k,j+1)}$, $S_1^{(k,j+1)}$, $S_2^{(k,j+1)}$, $D_{11}^{(k,j+1)}$, $D_{22}^{(k,j+1)} $, $D_{12}^{(k,j+1)} $ via \eqref{alg4:PolicyEvaluation+}.
                    \State \textbf{Policy Improvement}: Update the gain matrix $K_2^{(k,j)}$ to $K_2^{(k,j+1)}$:
                    \begin{align*}
                        K_2^{(k,j+1)} = - (R_{22,(k)} + D_{22}^{(k,j+1)})^{-1}S_2^{(k,j+1)}.
                    \end{align*}
                    \State \quad Increment inner counter: $j \gets j + 1$.
                \Until{$\| Z^{(k,j)} - Z^{(k,j-1)} \| < \varepsilon$}
            \State \quad Update $Z^{(k)} =Z^{(k,j)},\,P^{(k+1)}= P^{(k)} + Z^{(k)}$.
            \State \quad Update $R_{(k+1)}=\begin{bmatrix} R_{11,(k+1)} &R_{12,(k+1)}\\ R_{21,(k+1)}&R_{22,(k+1)} \end{bmatrix}=R_{(k)}+\begin{bmatrix} D_{11}^{(k+1,j)} &D_{12}^{(k+1,j)}\\ D_{12}^{(k+1,j)\top}&D_{22}^{(k+1,j)} \end{bmatrix}$, $R^{\sharp}_{22,(k+1)}=R_{11,(k)}-R_{12,(k)}R_{22,(k)}^{-1}R_{21,(k)}$, $\begin{bmatrix} K_{1}^{(k+1)} \\ K_{2}^{(k+1)} \end{bmatrix} =\begin{bmatrix} K_{1}^{(k)} \\ K_{2}^{(k)} \end{bmatrix}-R_{(k+1)}^{-1}\begin{bmatrix} S_{1}^{(k+1,j)} \\ S_{2}^{(k+1,j)} \end{bmatrix}$, and $M_{(k+1)} =-\bigl(S_{1}^{(k+1,j)}-R_{12,(k+1)}R_{22,(k+1)}^{-1}S_{2}^{(k+1,j)}\bigr)^{\top} R^{\sharp-1}_{22,(k+1)}\bigl(S_{1}^{(k+1,j)}-R_{12,(k+1)}R_{22,(k+1)}^{-1}S_{2}^{(k+1,j)}\bigr)$.
             \State \quad Increment outer counter: $k \gets k + 1$.
        \Until{$\| P^{(k)} - P^{(k-1)} \| < \varepsilon$}
        \State \Return $P^{(k)}$
    \end{algorithmic}
\end{algorithm}

\section{Convergence Analysis}
\label{sec:ConvergenceAnalysis}

We then proceed to establish the convergence proofs for Algorithms \ref{alg1:Nested_Iteration}, \ref{alg2:SemiModelBased}, \ref{alg3:SemiModelBased} and \ref{alg4:ModelFree} in Theorem \ref{alg:theorem}. To streamline the subsequent presentation, we first prove the convergence of the inner loop of Algorithm \ref{alg1:Nested_Iteration}. Different from the conclusions in \cite{Zhang2023} and \cite{Li2022}, we derive a more generalized result as follows.

\begin{lemma}
\label{alg:lemma}
For any $k =0,1,2,\cdots$, suppose 
\begin{itemize}
    \item $K_2^{(k,0)} \in \mathcal{K}[A_{(k)},C_{1,(k)},\cdots,C_{r,(k)}; B_2,D_{1,2},\cdots, D_{r,2}]$;
    \item There exist matrices $E_{l,(k)}(0 \leq l \leq r)$ such that $\sum_{l=0}^{r} E_{l,(k)}^\top E_{l,(k)} = M_{(k)}$ and the system
    \[
    \left[E_{0,(k)}, E_{1,(k)}, \dotsb, E_{r,(k)}; A_{(k)}, C_{1,(k)}, \dotsb, C_{r,(k)}\right]
    \]
    is \textit{stochastically detectable}.
\end{itemize}
Then, all the policies $\{K_2^{(k,j)}\}_{j=1}^\infty$ updated by \eqref{alg1:PolicyImprovement} are stabilizers of the system $[A_{(k)},C_{1,(k)},\cdots,C_{r,(k)}; B_2,D_{1,2},\\\cdots, D_{r,2}]$. This means $K_2^{(k,j)}\in \mathcal{K}[A_{(k)},C_{1,(k)},\cdots,C_{r,(k)}; B_2,D_{1,2},\cdots, D_{r,2}](j=0,1,2,\cdots)$. Moreover, the solution $Z^{(k,j+1)} \in \overline{\mathbb{S}}_{+}^n$ in Algorithm \ref{alg1:Nested_Iteration} is unique.
\end{lemma}

\begin{proof}
Since $K_2^{(k,0)} \in \mathcal{K}[A_{(k)},C_{1,(k)},\cdots,C_{r,(k)}; B_2,D_{1,2},\cdots, D_{r,2}]$ and $M_{(k)} + K^{(k,0)\top} R_{22,(k)} K^{(k,0)} \succeq 0$, Lyapunov recursion \eqref{alg1:PolicyEvaluation} has a unique solution $Z^{(k,1)} \in \overline{\mathbb{S}}_{+}^n$ by using Theorem 1 in \cite{Rami2000}. 

Suppose $j \geq 0$, $K_2^{(k,j)}\in \mathcal{K}[A_{(k)},C_{1,(k)},\cdots,C_{r,(k)}; B_2,D_{1,2},\cdots, D_{r,2}]$, and $Z^{(k,j+1)} \in \overline{\mathbb{S}}_{+}^n$ is the unique solution of Lyapunov recursion \eqref{alg1:PolicyEvaluation}. We now show $K_2^{(k,j+1)} \in \mathcal{K}[A_{(k)},C_{1,(k)},\cdots,C_{r,(k)}; B_2,D_{1,2},\cdots, D_{r,2}]$ and $Z^{(k,j+2)} \in \overline{\mathbb{S}}_{+}^n$. 

By using $K_2^{(k,j+1)} = - ( R_{22,(k)} + \sum_{l=1}^r D_{l,2}^\top Z^{(k,j+1)} D_{l,2} )^{-1} (B_2^\top Z^{(k,j+1)}  + \sum_{l=1}^r D_{l,2}^\top Z^{(k,j+1)} C_{l,(k)} )$, we can obtain
\begin{equation}
\label{alg:lemma_proof1}
\begin{aligned}
&Z^{(k,j+1)}(A_{(k)}+ B_2 K_2^{(k,j+1)})+(A_{(k)}+ B_2 K_2^{(k,j+1)})^{\top} Z^{(k,j+1)} \\
&+ \sum_{l=1}^r(C_{l,(k)} + D_{l,2} K_2^{(k,j+1)})^\top Z^{(k,j+1)}(C_{l,(k)} + D_{l,2} K_2^{(k,j+1)})\\
=&Z^{(k,j+1)}(A_{(k)}+ B_2 K_2^{(k,j)})+(A_{(k)}+ B_2 K_2^{(k,j)})^{\top} Z^{(k,j+1)} \\
&+ \sum_{l=1}^r(C_{l,(k)} + D_{l,2} K_2^{(k,j)})^\top Z^{(k,j+1)}(C_{l,(k)} + D_{l,2} K_2^{(k,j)})\\
&+Z^{(k,j+1)} B_2 (K^{(k,j+1)} - K^{(k,j)})+ (K_2^{(k,j+1)} - K_2^{(k,j)})^\top B^\top Z^{(k,j+1)}  \\
&+ \sum_{l=1}^r(C_{l,(k)} + D_{l,2} K_2^{(k,j+1)})^\top Z^{(k,j+1)}(C_{l,(k)} + D_{l,2} K_2^{(k,j+1)}) \\
&- \sum_{l=1}^r(C_{l,(k)} + D_{l,2} K_2^{(k,j)})^\top Z^{(k,j+1)}(C_{l,(k)} + D_{l,2} K_2^{(k,j)})  \\
=& - M_{(k)}-K_2^{(k,j)\top} R_{22,(k)} K_2^{(k,j)} - \sum_{l=1}^r(K_2^{(k,j+1)} - K_2^{(k,j)})^\top D_{l,2}^\top Z^{(k,j+1)} D_{l,2} (K_2^{(k,j+1)} - K_2^{(k,j)})\\
&+\sum_{l=1}^r K_2^{(k,j+1)\top} D_{l,2}^\top Z^{(k,j+1)} D_{l,2}( K_2^{(k,j+1)} -  K_2^{(k,j)})+\sum_{l=1}^r(K_2^{(k,j+1)} - K_2^{(k,j)})^\top D_{l,2}^\top Z^{(k,j+1)} D_{l,2} K_2^{(k,j+1)}\\
&- K_2^{(k,j+1)\top}(R_{22,(k)} + \sum_{l=1}^rD_{l,2}^\top Z^{(k,j+1)}D_{l,2}) ( K_2^{(k,j+1)} -  K_2^{(k,j)})\\
&-( K_2^{(k,j+1)} -  K_2^{(k,j)})^\top (R_{22,(k)} + \sum_{l=1}^rD_{l,2}^\top Z^{(k,j+1)}D_{l,2})K_2^{(k,j+1)}\\
=&- M_{(k)}-K_2^{(k,j)\top} R_{22,(k)} K_2^{(k,j)} - \sum_{l=1}^r(K_2^{(k,j+1)} - K_2^{(k,j)})^\top D_{l,2}^\top Z^{(k,j+1)} D_{l,2} (K_2^{(k,j+1)} - K_2^{(k,j)})\\
&- K_2^{(k,j+1)\top}R_{22,(k)} ( K_2^{(k,j+1)} -  K_2^{(k,j)})-( K_2^{(k,j+1)} -  K_2^{(k,j)})^\top R_{22,(k)}K_2^{(k,j+1)}\\
=&- M_{(k)}-K_2^{(k,j+1)\top} R_{22,(k)} K_2^{(k,j+1)}-( K_2^{(k,j+1)} -  K_2^{(k,j)})^\top (R_{22,(k)} + \sum_{l=1}^rD_{l,2}^\top Z^{(k,j+1)}D_{l,2})( K_2^{(k,j+1)} -  K_2^{(k,j)})
\end{aligned}
\end{equation}

Next, we set
\[
\hat{E}_{l,(k)} = 
\begin{bmatrix}
\mathbb{O}_{q \times n}^{\times (l)} \\
E_{l,(k)} \\
\mathbb{O}_{q \times n}^{\times (r-l)} \\
\mathbb{O}_{m_2 \times n}^{\times (l)} \\
\sqrt{\frac{1}{r+1}R_{22,(k,j+1)}}(K_2^{(k,j)} - K_2^{(k,j+1)})\\
\mathbb{O}_{m_2 \times n}^{\times (r-l)}\\
\sqrt{R_{22,(k)}}K_2^{(k,j+1)} \\
\end{bmatrix} \in \mathbb{R}^{[(q+m_2)(r+1)] \times n},\,0 \leq l \leq r ,
\]
where $R_{22,(k,j+1)}=R_{22,(k)} + \sum_{l=1}^rD_{l,2}^\top Z^{(k,j+1)}D_{l,2}$.

Subsequently, the \eqref{alg:lemma_proof1} can be converted to
\begin{equation*}
\label{alg:lemma_proof2}
\begin{aligned}
&Z^{(k,j+1)}(A_{(k)}+ B_2 K_2^{(k,j+1)}) +(A_{(k)}+ B_2 K_2^{(k,j+1)})^{\top} Z^{(k,j+1)} \\
&+ \sum_{l=1}^r(C_{l,(k)} + D_{l,2} K_2^{(k,j+1)})^\top Z^{(k,j+1)}(C_{l,(k)} + D_{l,2} K_2^{(k,j+1)})+\sum_{l=0}^{r} \hat{E}_{l,(k)}^\top \hat{E}_{l,(k)}=0.
\end{aligned}
\end{equation*}

Since the system $\left[E_{0,(k)}, E_{1,(k)}, \dotsb, E_{r,(k)}; A_{(k)} , C_{1,(k)} , \dotsb, C_{r,(k)} \right]$ is \textit{stochastically detectable}, there exists $\varTheta$ such that  
$
\left(A_{(k)}  + \varTheta E_{r,(k)}, C_{1,(k)}  + \varTheta E_{1,(k)}, \dotsb, C_{r,(k)} + \varTheta E_{r,(k)} \right)
$ 
is \textit{mean-square stable}.

Let \begin{small}$\hat{\varTheta}\in \mathbb{R}^{n \times [(q+m_2)(r+1)]}$\end{small} and $\hat{\varTheta} =$  
\[
\begin{bmatrix}
\varTheta^{\times (r+1)} & B_2\sqrt{(r+1)R_{22,(k,j+1)}^{-1}} & D_{1,2}\sqrt{(r+1)R_{22,(k,j+1)}^{-1}} & \dotsb & D_{r,2}\sqrt{(r+1)R_{22,(k,j+1)}^{-1}} &\mathbb{O}_{n \times m_2}
\end{bmatrix},
\]  
then the system  
\[
\left(A_{(k)}+ B_2 K_2^{(k,j+1)}  + \hat{\varTheta} \hat{E}_{0,(k)}, C_{1,(k)} + D_{1,2} K_2^{(k,j+1)} + \hat{\varTheta} \hat{E}_{1,(k)}, \dotsb, C_{r,(k)} + D_{r,2} K_2^{(k,j+1)} + \hat{\varTheta} \hat{E}_{r,(k)} \right)
\]  
is \textit{mean-square stable}.
By Lemma 3.5 in \cite{Wang2025}, we get the system  
\[
(A_{(k)}+ B_2 K_2^{(k,j+1)},\ C_{1,(k)} + D_{1,2} K_2^{(k,j+1)}, \dotsb, C_{r,(k)} + D_{r,2} K_2^{(k,j+1)})
\]  
is \textit{mean-square stable}.

This means $K_2^{(k,j+1)} \in \mathcal{K}[A_{(k)},C_{1,(k)},\cdots,C_{r,(k)}; B_2,D_{1,2},\cdots, D_{r,2}]$ and we get Lyapunov recursion \eqref{alg1:PolicyEvaluation} has a unique solution $Z^{(k,j+2)} \in \overline{\mathbb{S}}_{+}^n$ by using Theorem 1 in \cite{Rami2000}.
\end{proof}

Now, we prove the convergence of Algorithms \ref{alg1:Nested_Iteration}, \ref{alg2:SemiModelBased}, \ref{alg3:SemiModelBased}, and \ref{alg4:ModelFree}.

\begin{theorem}
\label{alg:theorem}
Assume the following conditions hold:
\begin{itemize}
    \item $R_{22} \succ 0$ and there exists an $L \in \mathcal{A}$.  
    \item There exist matrices $E_{l,(0)}(0 \leq l \leq r)$ such that $\sum_{l=0}^{r} E_{l,(0)}^\top E_{l,(0)} = Q - S^\top(0) R(0)^{-1}S(0)$ and the system
    \[
    \left[E_{0,(0)}, E_{1,(0)}, \dotsb, E_{r,(0)}; A_{(0)}, C_{1,(0)}, \dotsb, C_{r,(0)}\right]
    \]
is \textit{stochastically detectable}.
\end{itemize}
Then, we have:

(\romannumeral1). The sequences $\{P^{(k)}\}_{k\geq1}$ and $\{Z^{(k,j+1)}\}_{j\geq1}$ for $k=0,1,2,\cdots$ are well-defined via equations \eqref{alg1:PolicyEvaluation} and \eqref{alg1:outer_iteration} in Algorithm \ref{alg1:Nested_Iteration}. Moreover, when the corresponding rank conditions are satisfied, the sequences $\{P^{(k)}\}_{k\geq1}$ and $\{Z^{(k,j+1)}\}_{j\geq1}$ ($k=0,1,2,\cdots$) generated by Algorithms \ref{alg2:SemiModelBased}, (or \ref{alg3:SemiModelBased},or \ref{alg4:ModelFree}) are equivalent to their counterparts in Algorithm \ref{alg1:Nested_Iteration}.
And for each $k = 0,1,2,\cdots$ the following items are fulfilled:
\begin{enumerate}
    \item [$a_k$.] $\begin{bmatrix} K_1(P^{(k+1)})-K_1(P^{(k)})\\ L-K_2(P^{(k+1)})+K_2(0) \end{bmatrix} \in \mathcal{K} \left[ A_{(k)},C_{1,(k)},\cdots,C_{r,(k)}; \begin{bmatrix} B_1&B_2\end{bmatrix},\begin{bmatrix}D_{1,1}&D_{1,2}\end{bmatrix},\cdots, \begin{bmatrix}D_{r,1}&D_{r,2}\end{bmatrix}\right]$.
    \item [$b_k$.] $0 \preceq Z^{(k,j+1)} \preceq Z^{(k,j)}$, $\lim_{j\rightarrow\infty}Z^{(k,j)}=Z^{(k,*)} $ and $\lim_{j\rightarrow\infty}K_2^{(k,j)}=K_2^{(k,*)} \in \mathcal{K}[A_{(k)},C_{1,(k)},\cdots,C_{r,(k)}; B_2,\\D_{1,2},\cdots, D_{r,2}]$.
\end{enumerate}

(\romannumeral2). $\lim_{k\rightarrow\infty}P^{(k)}=P^*$, where $P^*$ is the stabilizing solution of the stochastic GTARE \eqref{zslqsdg:gtare}.
\end{theorem}

\begin{proof}
When the corresponding rank conditions are satisfied, we prove that the sequences $\{P^{(k)}\}_{k\geq1}$ and $\{Z^{(k,j+1)}\}_{j\geq1}$ ($k=0,1,2,\cdots$) generated by Algorithms \ref{alg2:SemiModelBased}, (or \ref{alg3:SemiModelBased},or \ref{alg4:ModelFree}) are equivalent to their counterparts in Algorithm \ref{alg1:Nested_Iteration}. For convenience of presentation, we assume that 
\[
K_2^{(k,j)}\in \mathcal{K}[A_{(k)},C_{1,(k)},\cdots,C_{r,(k)}; B_2,D_{1,2},\cdots, D_{r,2}]\,(k=0,1,2,\cdots;j=0,1,2,\cdots)
\] 
holds in Algorithm \ref{alg1:Nested_Iteration}. The detailed proof will be given in the latter part of this proof.

For Algorithm \ref{alg2:SemiModelBased}, we have $Z^{(k,j+1)}$ and $K_2^{(k,j)}$ ($k=0,1,2,\cdots;j=0,1,2,\cdots$) generated from \eqref{alg1:PolicyEvaluation} and \eqref{alg1:PolicyImprovement} in Algorithm \ref{alg1:Nested_Iteration} satisfy \eqref{alg2:PolicyEvaluation+} by using Lemma \ref{alg2:lemma}. When the data matrix has full column rank,  we can obtain the \eqref{alg2:PolicyEvaluation+} has a unique solution. This means policy improvement and policy evaluation in Algorithm \ref{alg2:SemiModelBased} is equivalent to solving its counterpart in Algorithm \ref{alg1:Nested_Iteration}.
For Algorithm \ref{alg3:SemiModelBased}, it can be seen from Lemma \ref{alg3:lemma} that $Z^{(k,j+1)}$ and $K_2^{(k,j)}$ ($k=0,1,2,\cdots;j=0,1,2,\cdots$) generated from \eqref{alg1:PolicyEvaluation} and \eqref{alg1:PolicyImprovement} in Algorithm \ref{alg1:Nested_Iteration} satisfy \eqref{alg3:PolicyEvaluation+}. Through Lemma 3 in \cite{Zhang2023}, we can conclude that when the rank condition of the Algorithm \ref{alg3:SemiModelBased} is guaranteed, \eqref{alg3:PolicyEvaluation+} has a unique solution. Naturally, policy improvement and policy evaluation in Algorithm \ref{alg3:SemiModelBased} is equivalent to solving its counterpart in Algorithm \ref{alg1:Nested_Iteration}.

For Algorithm \ref{alg4:ModelFree}, we have $Z^{(k,j+1)}$ and $K_2^{(k,j)}$ ($k=0,1,2,\cdots;j=0,1,2,\cdots$) generated from \eqref{alg1:PolicyEvaluation} and \eqref{alg1:PolicyImprovement} in Algorithm \ref{alg1:Nested_Iteration} satisfy \eqref{alg4:PolicyEvaluation+} by using Lemma \ref{alg4:lemma1}.  When the rank condition
\[
rank([\delta_{xx},\,\mathcal{I}_{xv_1},\,\mathcal{I}_{xv_2},\,\delta_{v_1v_1},\,\delta_{v_2v_2},\,\mathcal{I}_{v_2v_1}])=\frac{n(n+1)}{2}+m_1n+m_2n+\frac{m_1(m_1+1)}{2}+\frac{m_2(m_2+1)}{2}+m_1m_2,
\]
is satisfied,
we can obtain the \eqref{alg4:PolicyEvaluation+} has a unique solution. This means policy improvement and policy evaluation in Algorithm \ref{alg4:ModelFree} is equivalent to solving its counterpart in Algorithm \ref{alg1:Nested_Iteration}.

Then we prove the sequence $\{Z^{(k,j+1)}\}_{j\geq1}$ and $\{K_2^{(k,j)}\}_{j\geq1}$ in Algorithm \ref{alg1:Nested_Iteration} are well-defined and convergence for each $k = 0,1,2,\cdots$.

Since $L\in \mathcal{A}$, we have
\[K_2^{(0,0)}=L\in \mathcal{K}[A_{(0)},C_{1,(0)},\cdots,C_{r,(0)};B_2,D_{1,2},\cdots, D_{r,2}].\]
From Theorem 3.9 in \cite{Wang2025}, we can obtain
\[K_2^{(k+1,0)}=L-K_2(P^{(k+1)})+K_2(0) \in \mathcal{K}[A_{(k+1)},C_{1,(k+1)},\cdots,C_{r,(k+1)};B_2,D_{1,2},\cdots, D_{r,2}].\]
This also means 
\[\begin{bmatrix} K_1(P^{(k+1)})-K_1(P^{(k)})\\ L-K_2(P^{(k+1)})+K_2(0) \end{bmatrix} \in \mathcal{K} \left[ A_{(k)},C_{1,(k)},\cdots,C_{r,(k)}; \begin{bmatrix} B_1&B_2\end{bmatrix},\begin{bmatrix}D_{1,1}&D_{1,2}\end{bmatrix},\cdots, \begin{bmatrix}D_{r,1}&D_{r,2}\end{bmatrix}\right]\] by using proposition 3.1 in \cite{Wang2025}.

There exist matrices $E_{l,(k)}(0 \leq l \leq r)$ such that $\sum_{l=0}^{r} E_{l,(k)}^\top E_{l,(k)} = M_{(k)}$ and the system
    \[
    \left[E_{0,(k)}, E_{1,(k)}, \dotsb, E_{r,(k)}; A_{(k)}, C_{1,(k)}, \dotsb, C_{r,(k)}\right]
    \]
is \textit{stochastically detectable} from Theorem 3.9 in \cite{Wang2025}.
By Lemma \ref{alg:lemma}, we obtain all the policies $\{K_2^{(k,j)}\}_{j=1}^\infty$ updated by \eqref{alg1:PolicyImprovement} are stabilizers of the system $[A_{(k)},C_{1,(k)},\cdots,C_{r,(k)}; B_2,D_{1,2},\cdots, D_{r,2}]$. Moreover, the solution $Z^{(k,j+1)} \in \overline{\mathbb{S}}_{+}^n(j=0,1,\cdots)$ in Algorithm \ref{alg1:Nested_Iteration} is unique.

We combine \eqref{alg1:PolicyEvaluation} in Algorithm \ref{alg1:Nested_Iteration} and \eqref{alg:lemma_proof2} in Lemma \ref{alg:lemma} to obtain
\begin{align*}
& \Delta(A_{(k)}+ B_2 K_2^{(k,j+1)})+(A_{(k)}+ B_2 K_2^{(k,j+1)})^{\top} \Delta \\
&+ \sum_{l=1}^r(C_{l,(k)} + D_{l,2} K_2^{(k,j+1)})^\top \Delta(C_{l,(k)} + D_{l,2} K_2^{(k,j+1)})\\
&-(K_2^{(k,j)} - K^{(k,j+1)})^\top (R_{22,(k)} + \sum_{l=1}^rD_{l,2}^\top Z^{(k,j+1)}D_{l,2}) (K_2^{(k,j)} - K^{(k,j+1)})=0.
\end{align*}
where $\Delta=Z^{(k,j+2)}-Z^{(k,j+1)}$.
Since $K_2^{(k,j+1)} \in \mathcal{K}[A_{(k)},C_{1,(k)},\cdots,C_{r,(k)}; B_2,D_{1,2},\cdots, D_{r,2}]$, and $-(K_2^{(k,j)} - K^{(k,j+1)})^\top (R_{22,(k)} + \sum_{l=1}^rD_{l,2}^\top Z^{(k,j+1)}D_{l,2}) (K_2^{(k,j)} - K^{(k,j+1)}) \preceq 0$, we have $\Delta=Z^{(k,j+2)}-Z^{(k,j+1)} \preceq 0$.

Since the sequence $\{Z^{(k,j+1)}\}_{j\geq1}$ is monotonically decreasing and bounded below by 0, we can thus establish its convergence. Then we have $\lim_{j\rightarrow\infty}Z^{(k,j)}=Z^{(k,*)} $ and $\lim_{j\rightarrow\infty}K_2^{(k,j)}=K_2^{(k,*)}$.
From Theorem 3.9 in \cite{Wang2025}, we get $K_2^{(k,*)} \in \mathcal{K}[A_{(k)},C_{1,(k)},\cdots,C_{r,(k)}; B_2,D_{1,2},\cdots, D_{r,2}]$ and $\lim_{k\rightarrow\infty}P^{(k)}=P^*$, where $P^*$ is the stabilizing solution of the stochastic GTARE \eqref{zslqsdg:gtare}.
\end{proof}

\section{Simulation}
\label{sec:simulation}

Next, we will provide a concrete example to implement the reinforcement learning algorithm.

\begin{example}
Consider the following system matrices, input matrices and cost function parameters for a linear system:

\noindent\textbf{1. System Matrices and Input Matrices}
\[
\begin{aligned}
A &= \begin{bmatrix} 0 & -0.6 \\ 0.6 & -0.3 \end{bmatrix}, \quad
B_1 = \begin{bmatrix} 0.04 \\ 0.02 \end{bmatrix}, \quad
B_2 = \begin{bmatrix} 0.05 \\ 0.03 \end{bmatrix}, \\
\\
C &= \begin{bmatrix} -0.02 & 0.03 \\ -0.05 & 0.02 \end{bmatrix}, \quad
D_1 = \begin{bmatrix} 0.001 \\ 0.01 \end{bmatrix}, \quad
D_2 = \begin{bmatrix} 0.002 \\ 0.02 \end{bmatrix}.
\end{aligned}
\]

\noindent\textbf{2. Cost Function Parameters}
\[
\begin{aligned}
Q &= \begin{bmatrix} 0.5 & 0.0 \\ 0.0 & 0.6 \end{bmatrix}, \quad
S_1 = \begin{bmatrix} 0.01 & 0.02 \end{bmatrix}, \quad
S_2 = \begin{bmatrix} 0.03 & 0.04 \end{bmatrix}, \\
\\
R_{11} &= \begin{bmatrix} -0.2 \end{bmatrix}, \quad
R_{12} = \begin{bmatrix} 0.05 \end{bmatrix}, \quad
R_{21} = \begin{bmatrix} 0.05 \end{bmatrix}, \quad
R_{22} = \begin{bmatrix} 0.3 \end{bmatrix}.
\end{aligned}
\]

The solution obtained using Algorithm \ref{alg1:Nested_Iteration} is
\[
\begin{bmatrix}
1.9397 & -0.3990 \\
-0.3990 & 1.7771
\end{bmatrix}.
\]
The solution obtained using Algorithm \ref{alg2:SemiModelBased} is
\[
\begin{bmatrix}
1.9217 & -0.3965 \\
-0.3965 & 1.7685
\end{bmatrix}.
\]
The solution obtained using Algorithm \ref{alg3:SemiModelBased} is
\[
\begin{bmatrix}
1.9394 & -0.3991 \\
-0.3991 & 1.7780
\end{bmatrix}.
\]
The solution obtained using Algorithm \ref{alg4:ModelFree} is
\[
\begin{bmatrix}
2.0095 & -0.4149 \\
-0.4149 & 1.8217
\end{bmatrix}.
\]
\end{example}

\clearpage
\printbibliography

\end{document}